\documentclass[a4paper,draft,leqno,12pt]{article}

\usepackage{amsmath}
\usepackage{amscd}
\usepackage{amssymb}
\usepackage{enumerate}
\usepackage{indentfirst}
\usepackage{latexsym}
\usepackage{multicol}
\usepackage{pstcol,pst-fill,pstricks}
\usepackage{theorem}

\newcommand{\bv}{\boldsymbol{v}}
\newcommand{\bn}{\boldsymbol{n}}
\newcommand{\btau}{\boldsymbol{\tau}}
\newcommand{\bT}{\boldsymbol{T}}
\newcommand{\bF}{\boldsymbol{F}}
\newcommand{\bH}{\boldsymbol{H}}
\newcommand{\bphi}{\boldsymbol{\varphi}}
\newcommand{\bI}{\boldsymbol{I}}
\newcommand{\bw}{\boldsymbol{w}}
\newcommand{\bD}{\boldsymbol{D}}
\newcommand{\bA}{\boldsymbol{A}}
\newcommand{\bff}{\boldsymbol{f}}
\newcommand{\bomega}{\boldsymbol{\omega}}
\newcommand{\balpha}{\boldsymbol{\alpha}}
\newcommand{\bS}{\boldsymbol{S}}

\newcommand{\bPhi}{\boldsymbol{\Phi}}
\renewcommand{\rho}{\varrho}

\newcommand{\R}{\mathbb{R}}

\renewcommand{\div}{\operatorname{div}}

\newcommand{\rot}{\operatorname{rot}}
\newcommand{\pder}[2]{\frac{\partial #1}{\partial #2}}

\makeatletter\@addtoreset{equation}{section}\makeatother

\newtheorem{definition}{Definition}
\newtheorem{lemma}{Lemma}{\theorembodyfont{\rmfamily}
}\newtheorem{theorem}{Theorem}

\newenvironment{proof}{\textit{Proof. }}{\hfill$\Box$}

\begin{document}

\title{On the steady compressible Navier--Stokes--Fourier system 
\footnotetext{\textbf{Mathematics Subject Classification (2000). }
35Q30, 76N10\hfill\break}
\footnotetext{\textbf{Keywords. }Steady compressible Navier--Stokes--Fourier equations, 
slip boundary conditions, weak solutions, large data\hfill\break}}
\author{\textsc{Piotr B. Mucha$^1$} and \textsc{Milan Pokorn\'y$^2$}
}
\date{}
\maketitle

\begin{center}
{\small {

1. Institute of Applied Mathematics and Mechanics,Warsaw University}

{ul. Banacha 2, 02-097 Warszawa, Poland }

{E-mail: {\tt p.mucha@mimuw.edu.pl}}

\medskip

{2. Mathematical Institute of Charles University,}

{Sokolovsk\'a 83, 186 75 Praha 8, Czech Republic }

{E-mail: {\tt pokorny@karlin.mff.cuni.cz}}

}

\end{center}

\begin{abstract}
We study the motion of the steady compressible heat conducting viscous fluid in a bounded three dimensional domain
governed by the compressible Navier-Stokes-Fourier system. Our main result is the
existence of a weak solution to these equations for arbitrarily large data. A key element of the proof is a special approximation of the original system 
guaranteeing 
pointwise uniform boundedness of the density. Therefore the passage to the limit omits tedious technical tricks required by the standard
theory. Basic estimates on the solutions are possible to obtain by a suitable  choice of physically reasonable
boundary conditions.
\end{abstract}

\section{Introduction}

We consider the following system of partial differential equations describing the steady flow of a compressible heat 
conducting Newtonian fluid in a bounded three dimensional domain $\Omega$
\begin{equation}\label{1.1}
\displaystyle\div (\rho \bv)=0, 
\end{equation}
\vskip-1.1cm
\begin{equation} \label{1.2}
\displaystyle\div( \rho \bv \otimes \bv) - \div {\bf S}(\bv) +\nabla p(\rho, \theta)=\rho \bF,
\end{equation}
\vskip-1.1cm
\begin{equation} \label{1.3}
\displaystyle \div\big(\rho e(\rho,\theta) \bv\big) -\div \big(\kappa(\theta)\nabla \theta\big)=
{\bf S}(\bv):\nabla \bv - p (\rho,\theta)\div \bv, 
\end{equation}
where $\varrho: \Omega\to \R^+_0$ is the density of the fluid, $\bv:\Omega \to \R^3$ is the velocity field, 
$\bS(\bv) = 2\mu \bD(\bv) + \lambda (\div \bv) \bI$ is the viscous part of the stress tensor,
 $\bD(\bv) = \frac 12 (\nabla \bv + (\nabla \bv)^T)$ is the symmetric part of the velocity gradient, 
 $p(\cdot,\cdot): \R^+_0 \times \R^+ \to \R^+_0$, a given function, is the pressure, $\bF:\Omega \to R^3$  
  is the external force , $e(\cdot, \cdot): \R^+_0 \times \R^+ \to \R^+_0$, a given function, is the internal energy. 
  The system (\ref{1.1})-(\ref{1.3}) is known as the compressible Navier-Stokes--Fourier equations or the full Navier-Stokes system \cite{FeBook}.

We assume that the constitutive equation has the form
\begin{equation} \label{1.4}
p(\rho,\theta) = a_1\rho^\gamma + a_2 \rho \theta, \qquad a_1,a_2 >0,
\end{equation}
i.e. the pressure has one part corresponding to the ideal fluid and a so called elastic part; for more information see e.g. \cite{FeBook}. Even though we could consider more general pressure laws, we restrict ourselves to this simple model to avoid unnecessary technicalities in the proof. The corresponding internal energy takes the form
\begin{equation} \label{1.5}
e(\rho,\theta) = a_2 \theta + a_1 \frac{\rho^{\gamma-1}}{\gamma-1},
\end{equation}
see e.g. \cite{FeBook} or \cite{Batch}.
Note that in the full generality, the equation (\ref{1.3}) should be replaced by the conservation of the total energy, instead of conservation
of the internal energy only. For sufficiently regular class of solutions, including that  we are going to construct, the balance of the
kinetic energy is just a consequence of the momentum equation. We further simplify (\ref{1.3}). As our solution will be such that $\rho \in
L_\infty(\Omega)$ and $\bv \in W^1_p(\Omega)$, $p<\infty$, we get due to the fact that $\div (\rho\bv)=0$ in the weak sense (see
\cite{NoStBook})
$$
\div \big(\frac{1}{\gamma-1}\rho^\gamma \bv\big) = -\rho^\gamma \div\bv,
$$ 
again in the weak sense. Thus we may write instead of (\ref{1.3}) (we put $a_1=a_2=1$) the energy equation in the form 
\begin{equation} \label{1.3a}
\div\big(\rho \theta \bv\big) -\div \big(\kappa(\theta)\nabla \theta\big)=
{\bf S}(\bv):\nabla \bv - \rho \theta \div \bv. 
\end{equation}
The viscosity coefficients are for the sake of simplicity considered to be constant 
 such that the conditions of the thermodynamical stability
\begin{equation} \label{1.6}
\mu >0,  \qquad\lambda + \frac 23 \mu > 0
\end{equation}
are satisfied. Finally, the heat conductivity is assumed to be temperature dependent, i.e.
\begin{equation} \label{1.7}
\kappa(\theta) = a_3 (1+\theta^m), \qquad a_3,m>0.
\end{equation}
This fact is important for our study, we are not able to consider constant heat conductivity. Our domain $\Omega$ 
is sufficiently smooth, at least a $C^2$ domain. We supplement the system (\ref{1.1}), (\ref{1.2}) and (\ref{1.3a})
with the following  boundary conditions at $\partial \Omega$. 
For the velocity, we consider the slip boundary conditions
\begin{equation} \label {1.8}
\bv \cdot \bn = 0, \qquad \btau_k \cdot (\bT(p,\bv)\bn) + f \bv \cdot \btau_k=0 \quad \mbox{ at } \partial \Omega,
\end{equation}
where $\btau_k$, $k=1,2$ are two perpendicular tangent vectors to $\partial \Omega$, $\bn$ is the outer normal vector and 
$\bT(p,\bv) = -p \bI + \bS(\bv)$
is the stress tensor. The slip coefficient $f$ is non-negative (if $f=0$ we assume additionally that $\Omega$ is not rotationally symmetric). Recall that $f=0$ corresponds to the perfect slip while $f\to \infty$ leads to the homogeneous Dirichlet boundary conditions.

Concerning the temperature, we assume that 
\begin{equation} \label{1.10}
\kappa(\theta) \pder{\theta}{\bn} + L(\theta)(\theta-\theta_0) =0 \quad \mbox{ at } \partial \Omega,
\end{equation}
where $\theta_0:\partial \Omega \to \mathbb R^{+}$ is a strictly positive sufficiently smooth given function, say 
$\theta_0 \in C^2(\partial 
\Omega)$, $0< \theta_*\leq \theta_0 \leq \theta^* < \infty$ with $\theta_*,\theta^* \in \mathbb R^+$ and
\begin{equation} \label{1.11}
L(\theta) = a_4 (1+\theta^l), \qquad l \in R^+_0.
\end{equation}

We must also add the prescribed mass of the gas
\begin{equation} \label{1.12}
\int_\Omega \rho dx = M>0.
\end{equation}

The objective of this paper  is to prove  the  existence of weak solutions to problem (\ref{1.1})--(\ref{1.12}) for arbitrarily large 
data. 
Till now  only partial results have been proved (see e.g. \cite{BHNP}, \cite{LiBook2}, \cite{NoNoPo}, \cite{NoPa}) and only known general theorems concern weak solutions to 
the evolutionary version of the system \cite{FeBook}.
 The main obstacle was to construct 
suitable a priori estimates. Due to properties of the boundary conditions (\ref{1.10}) 
we are able to obtain  a nontrivial energy bound for weak solutions, 
saving the thermodynamical structure of the system. In the case of the barotropic gas we do not meet such difficulties. The energy bound
follows elementary from the momentum equation. Unfortunately, it is not the only difference. The standard methods introduced by
P.L. Lions \cite{LiBook2} do not work successfully  for the heat conducting case. However, a generalization of 
the technique  introduced in \cite{MuPo},\cite{PoMu} gives us sufficient tools to solve the stated problem.

An approach to 
system (\ref{1.1})--(\ref{1.12}) was considered in the book \cite{LiBook2}, unfortunately, 
this result can be viewed as conditional only, since instead of (\ref{1.12}) the author  assumed artificially that weak solutions
 satisfy  $\int_\Omega \rho^p dx = M^p$
for sufficiently large $p$. On the one hand, this condition is physically not acceptable, on the other hand, 
it simplifies considerably the mathematical analysis. Nevertheless, this result shows us  what is the difference in techniques for 
 the barotropic and heat conducting models.
 
 Looking on results concerning the classical solutions for problems with small data, we realize that the heat conducting system
 has the same mathematical  structure (difficulties) as the barotropic version of the model. Thus results from \cite{BHNP}, \cite{NoPa} 
 are almost immediately
  transformed  to the case of  the system (\ref{1.1})--(\ref{1.12}). For large data solutions the energy equation starts to play an
  important role, essentially changing the properties of the whole system.

The evolutionary case of the system (\ref{1.1})--(\ref{1.12}), under  general assumptions on the pressure law was considered in 
\cite{FeNo1} and \cite{FeNoPe}; 
however, the presented technique treats only the situation when the fluid is thermically isolated, i.e. $\pder{\theta}{n} =0$ at 
the boundary. It guarantees immediately the energy bound for weak solutions, but considering the limit $t\to \infty$, the only 
solution which can be obtained as the limit for large times 
(with time independent force) is the solution with constant temperature. This is connected to the fact that the model does not 
allow the heat transfer through the boundary and either the energy increases to infinity (non potential force) or the temperature 
approaches a constant value (potential force). Boundary condition (\ref{1.10}) allows the heat transfer through the boundary,
guaranteeing the balance of the total energy, and
 thus we are able to prove existence of solutions which are definitely nontrivial.

The main result of this paper is the following.

\begin{theorem} \label{t1}
Let $\Omega \in C^2$ be a bounded domain in $\mathbb R^3$.
Let $\bF \in L_\infty(\Omega)$ and 
 $$
  \gamma>3, \qquad m=l+1>\frac{3\gamma-1}{3\gamma-7}.
 $$
Then there exists a weak solution to (\ref{1.1})--(\ref{1.12}) such that 
$$
\varrho \in L_\infty(\Omega), \qquad \bv \in W^{1}_q(\Omega)
 \mbox { \ and \ } \theta \in W^{1}_q(\Omega)  \mbox{\ \ for all }1\leq q<\infty.
$$
\end{theorem}

The solution constructed by Theorem \ref{t1} is meant in the following sense.

\begin{definition} \label{d1}
The triple $(\rho,\bv,\theta)$ is a weak solution to (\ref{1.1})--(\ref{1.12}), if $\rho \in L_s(\Omega)$, 
$s\geq 2\gamma$, $\bv \in W^{1}_2(\Omega)$, $\theta \in W^{1}_2(\Omega)$ and 
$\theta^m \nabla \theta\in L_1(\Omega)$, $\bv \cdot \bn=0$ at $\partial \Omega$ in the sense of traces and
\begin{equation} \label{1.13}
\int_\Omega \varrho \bv \cdot \nabla \eta =0 \qquad  \forall \eta \in C^\infty(\overline{\Omega}),
\end{equation}
\vskip-.7cm
\begin{equation} \label{1.14} 
\begin{array}{c}
\displaystyle \int_\Omega \big(-\rho \bv\otimes \bv :\nabla \bphi + 2\mu \bD(\bv):\bD(\bphi) + \lambda\div \bv \div \bphi - 
p(\rho,\theta) \div \bphi \big) dx \\ \displaystyle + f \int_{\partial \Omega} (\bv\odot\btau) \cdot (\bphi \odot \btau) d\sigma = 
\int_\Omega \rho \bF \cdot \bphi dx \qquad \forall \bphi \in C^\infty(\overline\Omega); \bphi \cdot \bn = 0 \mbox{ at } \partial \Omega
\end{array}
\end{equation}
(we denoted by $\bv\odot \btau $ the vector $\bv-(\bv\cdot\bn)\bn$) and finally
\begin{equation} \label{1.15}
\begin{array}{c}
\displaystyle \int_\Omega \big(\kappa(\theta)\nabla \theta\cdot \nabla \psi - \rho \theta \bv \cdot \nabla \psi \big)dx + \int_{\partial \Omega} L(\theta)(\theta-\theta_0)\psi d\sigma \\
\displaystyle = \int_\Omega \big(2\mu|\bD(\bv)|^2 \psi dx + \lambda (\div \bv)^2 \psi -\rho\theta\div \bv \psi\big) dx 
\qquad \forall \psi \in C^\infty(\overline\Omega).
\end{array} 
\end{equation}
\end{definition}

The proof of Theorem \ref{t1} will be based on a special approximation procedure described in the next section
which is the kernel of our method. 
This section includes also  a priori estimates for the approximation. The structure of the approximative system
 gives us immediately the approximative 
density bounded uniformly in $L_\infty$, but we must prove refined $L_\infty$ estimates to verify 
that the limit solves the original system (\ref{1.1})--(\ref{1.3}). This idea has already been successfully applied in 
\cite{MuPo} and \cite{PoMu} in the case of barotropic flows.

The third  section contains a detailed proof of existence to the approximative system.
Here the main difficulty  comes from the energy equation, since the required positiveness of the temperature does not follow immediately.
In the next section we introduce an important quantity, the effective viscous flux and prove  its main properties, i.e. the compactness. 
This feature allows to improve information about the convergence of the density, which is the basic/fundamental fact in the theory
of the compressible Navier-Stokes equations.
 The last  section describes the refined
$L_\infty$ estimates for the approximative density and the passage to the limit. Then we prove that
the limit is indeed our sought solution in the meaning of Definition \ref{d1}.

As the reader may  easily check, our method works for slightly larger class of the pressure laws. It allows to consider e.g.  
$$
p(\rho,\theta) = p_b(\rho) + \rho\theta,
$$
where $p_b(\rho)$ is a strictly monotone function which behaves for large values as $\rho^\gamma$. The main steps of this
 generalization are similar to the barotropic case and can be found in \cite{PoMu}; since our problem is technically enough
  complicated, we shall avoid such generalizations.

Our new result is closely related to the barotropic version of the system (\ref{1.1})-(\ref{1.12}). Let us remind the state of the art in
this theory.
The steady compressible Navier--Stokes equations for arbitrarily large data were firstly successfully studied in the book 
\cite{LiBook2}, where, in the case of $p(\rho)=\rho^\gamma$ the existence of renormalized weak solutions was shown for 
$\gamma>1$ ($N=2$) and $\gamma\geq \frac 53$ ($N=3$) for Dirichlet boundary conditions. For potential forces with a small 
non potential perturbation the existence was improved in \cite{NoNo} for $\gamma >\frac 32$ ($N=3$). In the recent paper 
\cite{FrStWe1}  the authors proved the existence in two space dimensions also for $\gamma=1$.
 See also \cite{BrNo}, where the authors considered the three dimensional case and got existence
 for certain $\gamma$--s less than $\frac 53$, however, for periodic boundary conditions.  P.L. Lions also considered 
 the existence of solutions with locally  bounded density: for the case of Dirichlet boundary conditions he was able 
 to show their existence for $\gamma>1$ ($N=2$) and $\gamma\geq3$ ($N=3$). Nevertheless, to prove Theorem \ref{t1} the above methods are
 not sufficient, thus we present our new approach for the heat conducting model.
 
 Throughout the paper we use the standard notations for the Lebesgue, Sobolev, etc. spaces; generic constants are denoted by $C$ and 
 sequences $\epsilon \to 0$ always mean  suitable chosen subsequences $\epsilon_k \to 0^{+}$. For the sake of simplicity we put
 $a_1=a_2=a_3=a_4=1$.


\section{Approximation}

This section contains one of the main difficulties in the proof of Theorem \ref{t1} --- to find a good approximation of the problem 
(\ref{1.1})--(\ref{1.12}) for which we are able to show existence and prove the corresponding a priori estimates. We present  the approximative system as well as the proof of the fundamental a priori estimates. Next section deals then with the solvability of this system as well as with further a priori bounds.

Our approximative system will contain two parameters: a  number $\epsilon>0$ and an auxiliary function $K(\cdot)$ defined 
by a number $k>0$ as follows: 
\begin{equation} \label{def-K}
K(t)=\left\{
\begin{array}{lcr}
1 & \mbox{for} & t<k \\
\in [0,1] & \mbox{for} & k\leq t\leq k+1\\
0 & \mbox{for} & t>k+1;
\end{array}
\right.
\end{equation}
moreover we assume that $K'(t)<0$ for $t\in (k,k+1)$, where $k\in \R^+$. In the last section we pass with $\epsilon \to 0^+$ and we 
shall show that we may take $k$ sufficiently large such that $K(\varrho)\equiv 1$ for our solution. 
The approximation of our problem (\ref{1.1})--(\ref{1.12}) reads as follows
\begin{equation}\label{ap-ns}
\left.\begin{array}{r}
\displaystyle\epsilon \rho +\div(K(\rho)\rho \bv)-\epsilon \Delta \rho =\epsilon
h K(\rho) \\[9pt]
\displaystyle\frac{1}{2}\div(K(\rho)\rho \bv \otimes \bv)+\frac{1}{2}K(\rho)\rho \bv \cdot
\nabla \bv- \div {\bf S}(\bv)+\nabla  P(\rho,\theta)=\rho K(\rho)\bF \\[9pt]
\displaystyle -\div\Big((1+\theta^m)\frac{\epsilon +\theta}{\theta}\nabla \theta \Big)
+\div\Big(\bv\int_0^\rho K(t)dt\Big)\theta +\div\Big(K(\rho)\rho \bv\Big) \theta
\\[5pt]
\displaystyle
+K(\rho)\rho \bv \cdot \nabla \theta-\theta K(\rho) \bv\cdot \nabla \rho  = {{\bf S}(\bv):\nabla \bv}
\end{array}\right\}
\mbox{ in } \Omega,
\end{equation}
where 
\begin{equation}\label{ap-P}
P(\rho,\theta)=\int_0^\varrho \gamma t^{\gamma-1} K(t) dt +\theta \int_0^\rho K(t) dt = P_b(\rho) + \theta \int_0^\rho K(t) dt 
\end{equation}
and $h=\frac{M}{|\Omega|}$. 

Equation (\ref{ap-ns})$_3$ can be reformulated in the following way being the modification of the entropy equation:
\begin{equation}\label{ap-ent}
\begin{array}{c}
\displaystyle -\div\Big((1+\mbox{e}^{sm})\frac{(\epsilon +\mbox{e}^s)}{\mbox{e}^s}
\nabla s\Big)+K(\rho)\rho \bv \cdot \nabla s- K(\rho) \bv \cdot \nabla \rho 
+\div\Big(\bv \int_0^\rho K(t)dt \Big) 
\\[7pt]
\displaystyle
+\div\big(K(\rho)\rho \bv\big) = \frac{{{\bf S}(\bv):\nabla \bv}}{e^s}+
\frac{(1+\mbox{e}^{sm})(\epsilon+\mbox{e}^s)}{\mbox{e}^s}|\nabla s|^2 
 \mbox{ \ \ in }
\Omega,
\end{array}
\end{equation}
with the ''entropy'' $s$ defined as follows
\begin{equation}\label{en-ln}
s=\ln \theta.
\end{equation}
The distinguished entropy will allow to control the positiveness of the
temperature, what does not seem  to be elementary working directly with equation of type
(\ref{ap-ns})$_3$.

This system is completed by the boundary conditions at $\partial \Omega$
\begin{equation}\label{bc}
\begin{array}{c}
\displaystyle 
(1+\theta^m)(\epsilon +\theta)\pder{s}{\bn}
+L(\theta)(\theta-\theta_0)+\epsilon s=0,  \\[8pt]
\displaystyle 
\bv \cdot \bn = 0, \qquad \btau_k \cdot (\bT(p,\bv)\bn) + f \bv \cdot \btau_k = 0, \qquad k=1,2, \\[8pt]
\displaystyle \pder{\rho}{\bn} =0.
\end{array}
\end{equation}

A key element in the limit passage from the approximative problem to the original one is 
the energy estimate giving information independent of the choice of function
$K$, i.e. of the choice of the positive constant $k$ --- see (\ref{def-K}):

\begin{lemma}\label{ene-bound}
Suppose  solutions to (\ref{def-K})--(\ref{bc}) to be sufficiently smooth, i.e. $\rho$, $\bv$ and $\theta \in W^{2}_q(\Omega)$ 
for any $q<\infty$, $\theta>0$ 
 in $\Omega$. Let assumptions of Theorem \ref{t1} be satisfied. Then
\begin{equation}\label{e-b}
\begin{array}{c}
\displaystyle  0 \leq \rho \leq k, \qquad \int_\Omega \rho dx \leq M \mbox{ \ \ and \ \ }  \\[10pt]
\displaystyle||\bv||_{H^1(\Omega)}+||K(\rho)\rho||_{L_{2\gamma}(\Omega)}+
||P(\rho,\theta)||_{L_2(\Omega)}+||\theta||_{L_{3m}(\Omega)}+||\nabla \theta||_{L_r(\Omega)}
\\
\displaystyle+\int_{\partial \Omega} (e^s+e^{-s}) d \sigma + ||\nabla s||_{L_2(\Omega)}\leq 
C(||\bF||_{L_\infty(\Omega)},M),
\end{array}
\end{equation}
where the r.h.s. of (\ref{e-b}) is independent of $\epsilon$ and $k$, $s=\ln \theta$ and $r=\min\{2,\frac{3m}{m+1}\}$.
\end{lemma}

\begin{proof} The positiveness of the density and boundedness by $k$ follow elementary from features of function $K$ and the form of 
$(\ref{ap-ns})_1$. The integration of this equation leads to the bound on the total mass.
For details we refer to \cite{MuPo}. Let us prove the second part of (\ref{e-b}) which is definitely more complicated.
Multiply the approximative  momentum equation $(\ref{ap-ns})_2$  by $\bv$ and integrate it over $\Omega$:
\begin{equation}\label{b8a}
\begin{array}{c}
\displaystyle \int_\Omega \left(2\mu {\bf D}^2 (\bv)+\lambda \div^2\bv\right)dx +\int_{\partial \Omega}
f |\bv\odot \btau|^2d\sigma +
\int_\Omega \bv \cdot \nabla P_b(\rho) dx
\\[10pt]
\displaystyle =\int_\Omega K(\rho)
\rho  \bv \cdot \bF dx+\int_\Omega \Big(\int_0^\rho K(t) dt\Big) \,\theta \div \bv dx.
\end{array}
\end{equation}

To find a good form of the last term of the l.h.s. of (\ref{b8a}) we use the
approximative continuity equation $(\ref{ap-ns})_1$.
$$
\begin{array}{c}
\displaystyle \int_\Omega \bv \cdot \nabla P_b(\rho) dx=\frac{\gamma}{\gamma-1} \int_\Omega K(\rho)\rho \bv \cdot
\nabla \rho^{\gamma-1}dx
\\[10pt]
\displaystyle =-\frac{\gamma}{\gamma-1}\int_\Omega\left[
\epsilon \Delta \rho +\epsilon h K(\rho) -\epsilon \rho\right]
\rho^{\gamma-1}dx  \\[10pt]
\displaystyle
=\frac{\epsilon \gamma}{\gamma-1}\int_\Omega [\rho -hK(\rho)]
\rho^{\gamma-1}dx +\epsilon \gamma
\int_\Omega \rho^{\gamma-2}|\nabla \rho|^2dx.
\end{array}
$$
%
Thus the momentum equation gives the following inequality
\begin{equation}\label{b10a}
\begin{array}{c}
\displaystyle \int_\Omega \bS(\bv):\nabla \bv dx+ \int_{\partial \Omega}
f |\bv\odot \btau|^2d\sigma +\epsilon \gamma \int_\Omega
 \rho^{\gamma-2}|\nabla \rho|^2dx +  \frac{\epsilon \gamma}{\gamma-1} \int_\Omega \rho^\gamma dx \\[10pt]
\displaystyle -\int_\Omega \Big(\int_0^\rho K(t) dt\Big)\,
 \theta \div \bv dx 
\leq C\Big(1+\int_\Omega |K(\rho)\rho \bv  \cdot \bF|dx
\Big).
\end{array}
\end{equation}
Integrating the energy equation $(\ref{ap-ns})_3$ and employing the boundary condition (\ref{bc})$_1$ we get
\begin{equation}\label{b1}
\begin{array}{c}
\displaystyle
\int_{\partial \Omega} \big(L(\theta)(\theta-\theta_0)+\epsilon s\big) d \sigma=
\int_\Omega \Big(\bS(\bv):\nabla \bv - \big(\int_0^\rho K(t)dt\big) \theta \div \bv\Big) dx, 
\end{array}
\end{equation}
since the integration by parts gives the following identity
$$
\begin{array}{c}
\displaystyle
\int_\Omega \Big[ K(\rho)\rho \bv \cdot \nabla \theta- \theta K(\rho) \bv \cdot \nabla \rho 
+\div\Big(\bv \int_0^\rho K(t)dt \Big)\theta 
\\[7pt]
\displaystyle
+\div\big(K(\rho)\rho \bv\big)  \theta\Big] dx
=\int_\Omega \Big(\int_0^\rho K(t)dt\Big) \theta \div \bv dx.
\end{array}
$$

Summing up (\ref{b10a}) and (\ref{b1}) we get
\begin{equation}\label{b2}
\begin{array}{c}
\displaystyle
\int_{\partial \Omega} \big(L(\theta) \theta +\epsilon s^+\big)d \sigma +\epsilon \gamma \int_\Omega
 \rho^{\gamma-2}|\nabla \rho|^2dx +  \frac{\epsilon\gamma}{\gamma-1} \int_\Omega \rho^\gamma dx \\[6pt] 
 \displaystyle \leq \int_{\partial \Omega} 
\epsilon s^- d \sigma + C\Big(1+\int_\Omega |K(\rho)\rho \bv \cdot\bF |dx\Big),
\end{array}
\end{equation}
where $s^+$ and $s^-$ are the positive and negative parts of the entropy, respectively ($s=s^+-s^-$).

We shall concentrate the attention on the first term of the r.h.s. of (\ref{b2}). Note that the control of
the negative part of entropy $s$ is not immediate.
We integrate the entropy equation (\ref{ap-ent}) over $\Omega$ getting
\begin{equation} \label{2.14}
\begin{array}{c}
\displaystyle
\int_{\partial \Omega} \left[ \frac{L(\theta)(\theta-\theta_0)}{\theta}+\epsilon s e^{-s}\right]
d \sigma+
\int_\Omega \Big(K(\rho)\rho \frac{\bv \cdot\nabla \theta}{\theta}  -K(\rho)\bv \cdot\nabla \rho\Big) dx \\ [9pt]
\displaystyle
= \int_\Omega \left[ \frac{{\bf S}(\bv):\nabla \bv}{\theta}
 +\frac{(1+\theta^m)(\epsilon+\theta)}{\theta}|\nabla s|^2\right]dx.
\end{array}
\end{equation}
So
\begin{equation}\label{b3}
\begin{array}{c}
\displaystyle
\int_\Omega \left( \frac{{\bf S}(\bv): \nabla \bv}{\theta}+
\frac{(1+\theta^m)(\epsilon+\theta)}{\theta}|\nabla s|^2\right)dx+
\int_{\partial \Omega}\left(\frac{L(\theta)\theta_0}{\theta}+\epsilon |s^-| e^{|s^-|} \right)d \sigma
\\[8pt]
\displaystyle -
\int_\Omega K(\rho)\rho \bv \cdot \nabla (s -\ln \rho) dx
\leq \int_{\partial \Omega} L (\theta)d\sigma + \int_{\partial \Omega} \epsilon s^+ \mbox{e}^{-s^+} d\sigma.
\end{array}
\end{equation}
Let us look closer at the last term in the l.h.s. of (\ref{b3}). We have
\begin{equation} \label{2.16}
\begin{array}{c}
\displaystyle -\int_\Omega K(\rho)\rho \bv \cdot \nabla (s -\ln \rho) dx=
\int_\Omega K(\rho)\rho \bv \cdot \nabla \ln \rho dx
-\int_\Omega K(\rho)\rho \bv \cdot\nabla s dx =I_1+I_2,
\end{array}
\end{equation}
and employing  $(\ref{ap-ns})_1$ we get

\begin{equation} \label{2.17}
\begin{array}{c}
\displaystyle
\int_\Omega K(\rho)\rho \bv \cdot \nabla \ln \rho dx=- \int_\Omega \div (K(\rho) \rho \bv) \ln \rho dx\\
\displaystyle =\int_\Omega \big(-\epsilon \Delta \rho+\epsilon \rho - \epsilon h K(\rho)\big)\ln \rho dx
=\int_\Omega \Big( \epsilon \frac{|\nabla \rho|^2}{\rho} - \epsilon h K(\rho) \ln \rho +
\epsilon \rho \ln \rho \Big) dx.
\end{array}
\end{equation}
The first term has a good sign, the second term has a good sign for $\rho \leq 1$, too, and
for $\rho \geq 1$ is easily bounded by $\epsilon h \rho$. Similarly, the last term can be controlled by the term $\epsilon \int_\Omega \rho^\gamma dx$. The proof was rather formal, as we do not know whether $\rho>0$ in $\Omega$.  However, we may write $K(\rho) \bv \cdot \nabla (\rho +\delta)$ in (\ref{2.14}) with $\delta>0$ and find an analogue of (\ref{2.17}) with $\ln(\rho +\delta)$. Finally we pass with $\delta\to 0^+$ and get precisely the same information as above. Next
\begin{equation} \label{2.18}
\begin{array}{c}
\displaystyle
I_2=-\int_\Omega K(\rho)\rho \bv \cdot \nabla s dx = 
\int \big(\epsilon \Delta \rho-\epsilon \rho + \epsilon h K(\rho)\big)s dx \\ 
\displaystyle
=\int_\Omega \big( -\epsilon \nabla \rho \nabla s - \epsilon \rho \ln \theta+
\epsilon h K(\rho) \ln \theta\big) dx.
\end{array}
\end{equation}
Considering the r.h.s. of (\ref{2.18}), we have 
\begin{equation} \label{2.19}
\begin{array}{c}
\displaystyle \left|\epsilon \int \nabla \rho \nabla s dx \right|
\leq \epsilon \|\nabla \rho\|_{L_2(\Omega)} \|\nabla s\|_{L_2(\Omega)} \\
\displaystyle \leq \frac 14 \epsilon \Big(\int_\Omega \frac{|\nabla \rho|^2}{\rho} dx +
 \int_\Omega |\nabla \rho|^2 \rho^{\gamma-2} dx \Big) + \frac 14 \|\nabla s\|_{L_2(\Omega)}^2.
\end{array}
\end{equation}
%
Moreover, 
$\int_\Omega -\epsilon \rho \ln \theta dx$
has a good sign for $\theta \leq 1$ and for $\theta>1$
\begin{equation} \label{2.20}
\int_\Omega -\epsilon \rho (\ln \theta)^+ dx \leq \epsilon \|\rho\|_{L_2(\Omega)} \|s^+\|_{L_2( \Omega)}  \leq \frac \epsilon 4
\big(\|s^+\|_{L_1(\partial\Omega)} + \|\nabla s\|_{L_2(\Omega)}\big) + \frac{\epsilon}{4} \|\rho^\gamma\|_{L_1(\Omega)} + C.
\end{equation}
The last term of (\ref{2.18}) can be treated as follows (one part has again a good sign)
\begin{equation} \label{2.21}
\int \epsilon h K(\rho) |(\ln \theta)^-| dx \leq C \epsilon \int |s^-| dx \leq C+ \frac{1}{2} \int_{\partial 
\Omega} \epsilon |s^-| e^{|s^-|} d\sigma + \frac 14 ||\nabla s||_{L_2(\Omega)}.
\end{equation}

Then combining (\ref{b3}) with inequality (\ref{b2}) and with (\ref{2.17})--(\ref{2.21})  we obtain
\begin{equation}\label{b4}
\int_\Omega \left(\frac{{\bf S}(\bv): \nabla \bv}{\theta}+
\frac{1+\theta^m}{\theta^2}|\nabla \theta|^2\right)dx+
\int_{\partial \Omega}\left(L(\theta)\theta+
\frac{L(\theta)\theta_0}{\theta}+\epsilon|s|\right)
d \sigma\leq H,
\end{equation}
where
%
$$
H=C\Big(1+\int_\Omega |K(\rho) \rho \bv \bF|dx\Big).
$$
Thus from the growth conditions we deduce the following 
``homogeneous'' estimates:
$$
\begin{array}{c}
\displaystyle\left(\int_{\partial \Omega} 
\theta^{l+1}d\sigma\right)^{1/(l+1)}\leq  H^{1/(l+1)},\qquad
\displaystyle\left(\int_\Omega |\nabla \theta^{m/2}|^2\right)^{1/m}\leq 
H^{1/m}.
\end{array}
$$
%
To obtain a good information about integrability of the temperature we use
the following Poincar\'e type inequality
$$
\left(\int_\Omega |\theta^{m/2}|^2dx\right)^{1/m}
\leq C(\Omega)\left( \left(\int_\Omega |\nabla \theta^{m/2}|^2dx\right)^{1/m}+
\left(\int_{\partial\Omega} \theta^{l+1}d\sigma\right)^{1/(l+1)}\right)
$$
which can be proved elementary. Then
 the imbedding theorem leads to the bound
\begin{equation}\label{b7}
\left(\int_\Omega \theta^{3m}dx\right)^{1/3m} \leq H^{1/m}+
H^{1/(l+1)}.
\end{equation}
To simplify further calculations, we set $l+1=m$. Note that we may allow also different values of $l$, however, for the prize that the further calculations become more technical which we try to avoid.

We return to (\ref{b10a}). H\"older's inequality yields\footnote{Note that we used Korn's inequality; for $f=0$ we therefore require that $\Omega$ is not rotationally symmetric, for more details see \cite{NoStBook}.} 
\begin{equation}\label{b10}
\begin{array}{c}
\displaystyle ||\bv||_{H^1(\Omega)}^2+\epsilon \gamma \int_\Omega \rho^{\gamma-2}|\nabla \rho|^2dx + \frac{\epsilon\gamma}{\gamma-1} \int_\Omega \rho^\gamma dx  \\
\displaystyle \leq C\left(1+\int_\Omega |K(\rho)\rho \bv  \cdot \bF|dx+
\int_\Omega |\theta \int_0^\rho K(t)dt|^2dx\right).
\end{array}
\end{equation}

The next step of our estimation is the bound on $P_b(\rho)$
which is necessary to estimate the r.h.s. of (\ref{b10}). 
We just repeat the
method for the barotropic case, but here we shall obtain an extra term related to
the temperature.

Introduce $\bPhi: \Omega \rightarrow \R^3$ defined as a solution to the following
problem
\begin{equation}\label{b11}
\begin{array}{lcr}
\div \bPhi =P_b(\rho)-\{P_b(\rho)\}& \mbox{in} & \Omega, \\
\bPhi =\boldsymbol{0} &{\rm at } & \partial \Omega,
\end{array}
\mbox{ \ \ \ with \ } \{P_b(\rho)\} =\frac{1}{|\Omega|} \int_\Omega 
P_b(\rho)dx.
\end{equation}
The basic theory to the stationary Stokes system gives 
the existence of a vector field satisfying (\ref{b11}) with the following
estimate for a solution to (\ref{b11}) (for another possible proof, using directly estimates of special solutions to system (\ref{b11}), see \cite{NoStBook})
\begin{equation}\label{b12}
||\bPhi||_{H^1_0(\Omega)}\leq C||P_b||_{L_2(\Omega)}.
\end{equation}

From the structure of $P_b(\rho)$ and information that 
$\int_\Omega \rho_\epsilon dx \leq M$ 
 we easily get applying the interpolation inequality
$$
\{P_b(\rho)\}\leq \delta ||P_b(\rho)||_{L_2(\Omega)}+C(\delta,M)
\mbox{ \ \ \ \ \ \ for any $\delta>0$.}
$$
%
Multiplying the momentum equation $(\ref{ap-ns})_2$ by $\bPhi$, employing (\ref{b10}) and (\ref{b12}),
 we conclude after standard estimates of the r.h.s to $(\ref{ap-ns})_2$
\begin{equation}\label{b14}
||P_b(\rho)||_{L_2(\Omega)}^2\leq C\left(1+\int_\Omega 
|K(\rho)\rho \bv \otimes \bv|^2dx
+\int_\Omega |\theta\int_0^\rho K(t)dt |^2dx\right).
\end{equation}
As
\begin{equation}\label{b15}
||P_b(\rho)||_{L_2(\Omega)}^2\geq C \left(\int_\Omega (K(\rho)\rho)^{2\gamma}dx
+\int_\Omega \left(\int_0^\rho K(t)dt\right)^{2\gamma}dx\right),
\end{equation}
recalling that $2\gamma >6$, we get a bound for the first integral in the r.h.s.
of (\ref{b14})
\begin{equation}\label{b16}
\begin{array}{c}
\displaystyle \int_\Omega |K(\rho)\rho \bv \otimes \bv|^2 dx\leq
c||\bv||^4_{H^1(\Omega)}||K(\rho)\rho||_{L_{6}(\Omega)}^2 \\ \displaystyle \leq c||\bv||^4_{H^1(\Omega)}||K(\rho)\rho||^{\frac{2(\gamma-3)}{3(2\gamma-1)}}_{L_{1}(\Omega)} ||K(\rho)\rho||^{\frac{10\gamma}{3(2\gamma-1)}}_{L_{2\gamma}(\Omega)}  \leq
\delta ||P_b(\rho)||^2_{L_2(\Omega)}+
C(\delta, M)||\bv||^{\frac{6(2\gamma-1)}{3\gamma-4}}_{H^1(\Omega)}.
\end{array}
\end{equation}
Hence a suitable choice of $\delta$ in (\ref{b16}) simplifies (\ref{b14}) to
\begin{equation}\label{b17}
||P_b(\rho)||_{L_2(\Omega)}^2\leq C\left(1+
||\bv||^{\frac{6(2\gamma-1)}{3\gamma-4}}_{H^1(\Omega)}+
\int_{\Omega} |\theta\int_0^\rho K(t)dt |^2dx\right).
\end{equation}
The last integral can be viewed by (\ref{b15}) in the form 
\begin{equation}\label{b18}
||\int_0^\rho K(t)dt||_{L_{2\gamma}(\Omega)}+
||K(\rho)\rho||_{L_{2\gamma}(\Omega)}\leq C\left(1+||\bv||_{H^1(\Omega)}^{\frac{3}{\gamma}\frac{2\gamma-1}{3\gamma-4}}
+\Big(\int_\Omega |\theta\int_0^\rho K(t)dt|^2dx\Big)^{\frac{1}{2\gamma}}\right).
\end{equation}
Within our estimation we concentrate on a precise specification of powers of
norms. Then, due to our growth conditions we shall be able to construct
the desired bound (\ref{e-b}).

The last integral in (\ref{b18}) can be treated as follows (we need $m>\frac 23$ and $m> \frac{2\gamma}{3(\gamma-1)}$)
\begin{equation}\label{b19}
\begin{array}{c}
\displaystyle ||\theta\int_0^\rho K(t)dt ||_{L_2(\Omega)}^{1/\gamma}\leq 
||\theta||_{L_{3m}(\Omega)}^{1/\gamma}||\int_0^\rho K(t)dt||_{L_{\frac{6m}{3m-2}}(\Omega)}^{1/\gamma} 
 \\[6pt] 
 \displaystyle \leq \|\theta\|_{L_{3m}(\Omega)}^{\frac{1}{\gamma}} 
 ||\int_0^\rho K(t)dt||_{L_{1}(\Omega)}^{\frac{(3m-2)\gamma-3m}{3m\gamma(2\gamma-1)}} 
 ||\int_0^\rho K(t)dt||_{L_{2\gamma}(\Omega)}^{\frac{3m+2}{3m(2\gamma-1)}},
\end{array}
\end{equation}
so (\ref{b18}) and (\ref{b19})
  with the H\"older inequality imply
$$
||\int_0^\rho K(t)dt||_{L_{2\gamma}(\Omega)}+
||K(\rho)\rho||_{L_{2\gamma}(\Omega)}\leq C\Big(1+||\bv||_{H^1(\Omega)}^{\frac 3 \gamma \frac {2\gamma-1}{3\gamma-4}}
+||\theta||_{L_{3m}(\Omega)}^{\frac {3m}{\gamma} \frac{2\gamma-1}{6m(\gamma-1)-2}}\Big).
$$
Applying the inequality for the temperature --- (\ref{b7}) --- we obtain (recall that we put $l+1=m$)
\begin{equation}\label{b21}
||\int_0^\rho K(t)dt||_{L_{2\gamma}(\Omega)}+
||K(\rho)\rho||_{L_{2\gamma}(\Omega)}\leq C\left(1+||\bv||_{H^1(\Omega)}^{\frac 3 \gamma \frac {2\gamma-1}{3\gamma-4}}
+H_1^{\frac {3}{\gamma} \frac{2\gamma-1}{6m(\gamma-1)-2}}
\right).
\end{equation}
We have to estimate $H$; it holds
$$
\int_\Omega |K(\rho)\rho \bv \bF|dx\leq ||\bv ||_{L_6(\Omega)}
||K(\rho)\rho||_{L_{6/5}(\Omega)}
||\bF||_{L_\infty(\Omega)}.
$$
Using the interpolation between $1$ and $2\gamma$ as above leads
 to the following bound 
\begin{equation}\label{b23}
\int_\Omega |K(\rho)\rho \bv \bF|dx\leq C(M)||\bv||_{H^1(\Omega)}||K(\rho)\rho
||_{L_{2\gamma}(\Omega)}^{\frac{\gamma}{3(2\gamma-1)}}.
\end{equation}

Inserting this inequality to the r.h.s. of (\ref{b21}), recalling that
$m\geq \frac 14$ and applying the standard H\"older inequality we obtain
 from (\ref{b21}) estimate on the density 
\begin{equation}\label{b24}
||\int_0^\rho K(t)dt||_{L_{2\gamma}(\Omega)}+
||K(\rho)\rho||_{L_{2\gamma}(\Omega)}\leq C\Big(1+
||\bv||_{H^1(\Omega)}^{\frac 3 \gamma \frac {2\gamma-1}{3\gamma-4}} + 
||\bv||_{H^1(\Omega)}^{\frac{1}{\gamma}\frac {2\gamma-1}{2m(\gamma-1)-1}}\Big).
\end{equation}

As we can see later, the first term is the most restrictive. So by (\ref{b23}) and (\ref{b24}) 
we conclude (for $m>\frac{3\gamma-1}{6\gamma-6}$)
\begin{equation}\label{b25}
\int_\Omega |K(\rho)\rho \bv \bF|dx\leq
C\Big(1+||\bv||_{H^1(\Omega)}^{\frac{3\gamma-3}{3\gamma-4}}\Big).
\end{equation}
Hence  we obtain from (\ref{b7})
\begin{equation}\label{b26}
||\theta||_{L_{3m}(\Omega)}\leq
C\left(1+||\bv||_{H^1(\Omega)}^{\frac{1}{m}\frac{3\gamma-3}{3\gamma-4}}
\right).
\end{equation}
From (\ref{b19}) we easily see that
\begin{equation}\label{b27}
||\theta\int_0^\rho K(t)dt ||_{L_2(\Omega)}\leq C||\theta||_{L_{3m}(\Omega)}
||\int_0^\rho K(t)dt||^{\frac{3m+2}{3m}\frac{\gamma}{2\gamma-1}}_{L_{2\gamma}(\Omega)}.
\end{equation}

Summing up inequalities (\ref{b10}), (\ref{b25}) and (\ref{b27})
 we obtain the main bound on the norm of the velocity
$$
||\bv ||_{H^1(\Omega)}^2\leq C\left(1+
||\bv||^{\frac{3\gamma-3}{3\gamma-4}}_{H^1(\Omega)}
+||\bv||^{\frac{2}{m}\frac{3\gamma-3}{3\gamma-4}+
\frac{2}{m}\frac{3m+2}{3\gamma-4}}_{H^1(\Omega)}
\right).
$$

The above bound implies the a priori bound
\begin{equation}\label{b29}
||\bv||_{H^1(\Omega)}\leq C(||\bF||_{L_\infty},M),
\end{equation}
provided suitable dependence between $\gamma$ and $m$ holds, which can be
described by the sufficient condition ($\gamma>3$)
\begin{equation}\label{b30}
 m > \frac{3\gamma-1}{3\gamma-7}.
\end{equation}

Note that as we take $\gamma$ near $3$ then $m>4$ and for $\gamma=4$ we have $m>\frac {11}{5}$. 
Moreover, the above needed conditions $m>\frac{3\gamma-1}{6\gamma-1}$, $m >\frac 23$ and $m>\frac{2\gamma}{3(\gamma-1)}$ are clearly less restrictive than (\ref{b30}).

Bound (\ref{b29}) implies immediately the a priori estimate (\ref{e-b}), since it
follows from (\ref{b4}) with (\ref{1.11}), (\ref{b17}), (\ref{b24})--(\ref{b27}), together with (\ref{3.13}). 
\end{proof}

\section{Existence for the approximative system}

The aim of this section is to show that for any $\epsilon >0$ and $k>0$ there is a solution to
the approximative system  (\ref{ap-ns})--(\ref{bc}). 
We prove

\smallskip

\begin{theorem} \label{t2}
Let the assumptions of Theorem \ref{t1} be satisfied. Moreover, let $\epsilon >0$ and $k>0$. Then there exists a strong solution
$(\rho,\bv,s)$ to (\ref{ap-ns})  such that
$$
\rho \in W^2_p(\Omega), \;\; \bv \in W^2_p(\Omega) \mbox{ \ \ 
and \ \ } s \in W^2_p(\Omega) \mbox{ \ \ for \ \ }
1\leq p <\infty.
$$
Moreover 
$0 \leq \rho \leq k$ in $\Omega$, 
$\int_\Omega \rho dx \leq M$ and
\begin{equation}\label{b-m}
||\bv||_{W^1_{3m}(\Omega)}+\sqrt{\epsilon}||\nabla \rho||_{L_2(\Omega)} + \|\nabla \theta\|_{L_r(\Omega)} + \|\theta\|_{L_{3m}(\Omega)} \leq C(k),
\end{equation}
where $\theta = \mbox{e}^s$, $r=\min\{2,\frac{3m}{m+1}\}$ and the r.h.s. of (\ref{b-m}) is independent of the parameter  $\epsilon$.
\end{theorem}

The proof of the existence to the approximative system (\ref{ap-ns}) will follow from the standard application 
of the Leray-Schauder fixed point theorem. It will be split into several lemmas. First we consider the continuity equation. We denote for 
$p \in [1,\infty]$ 
$$
M_p=\{ \bw \in W^2_p(\Omega); \bw \cdot \bn =0 \mbox{ at } \partial \Omega\}.
$$
We have

\begin{lemma} \label{l2}
Let $q>3$. Then the operator 
$$
S:M_q \to W^2_p(\Omega) \mbox{ \ \ \ \ \ for \ } 1 \leq p < \infty 
$$
such that 
$
S(\bv)=\rho$, where $\rho$ is the solution to the following problem
\begin{equation}\label{c1}
\begin{array}{lcr}
\displaystyle \epsilon \rho -\epsilon \Delta \rho= \epsilon h K(\rho) - \div(K(\rho)\rho \bv) 
&  \mbox{in} & \Omega, \\[8pt]
\displaystyle \frac{\partial \rho}{\partial \bn}=0 & \mbox{at} & \partial \Omega
\end{array}
\end{equation}
is a well defined continuous compact operator from $M_q$ to $W^2_p(\Omega)$, $1\leq  p< \infty$.
 In particular, the solution to (\ref{c1}) is unique. Moreover
 \begin{equation} \label{e1}
 \|\rho\|_{W^l_p(\Omega)} \leq C(k,\epsilon) (\|\bv\|_{W^{l-1}_p(\Omega)}+1), \qquad l=1,2.
 \end{equation}
 \end{lemma}
 
\begin{proof}
It follows from \cite{MuPo}, Proposition 3.1 (there, the two dimensional case was considered). See also \cite{NoStBook}. 
\end{proof}

Next, we define the operator 
$$
{\cal T}: M_p \times W^2_p(\Omega) \to M_p \times W^2_p(\Omega)
\mbox{ \ \ such that \ \ }
{\cal T}(\bv,s)=(\bw,z) ,
$$
where $(\bw,z)$ is the solution to the following system
\begin{equation}\label{L-S}
\begin{array}{c}
\left.
\begin{array}{r}
\displaystyle
-\div {\bf S}(\bw)=-\frac 12 \div(K(\rho)\rho \bv \otimes \bv)
-\frac 12 K(\rho)\rho \bv \cdot \nabla \bv - \nabla P(\rho,e^s) + K(\rho)\rho \bF 
 \\[16pt]
\displaystyle
-\div\left((1+\mbox{e}^{ms})(\epsilon+\mbox{e}^s)\nabla z\right)=
{\bf S}(\bv):\nabla \bv   
-\div \Big(\bv \int_0^\rho K(t)dt\Big)\mbox{e}^s\\[4pt]
\displaystyle -\div\big(K(\rho)\rho \bv\big)\mbox{e}^s-
\mbox{e}^s K(\rho)\rho \bv \cdot \nabla s+ \mbox{e}^s K(\rho) \bv \cdot \nabla \rho  
\end{array}
\right\} \mbox{ in } \Omega,
\\[40pt]
\left.
\begin{array}{r}
 \displaystyle
\bw\cdot \bn=0,\;\;\;
 \bn \cdot {\bf S}(\bw)\cdot \btau_l+f\bw \cdot \btau_l=0 \mbox{ \ \ for } l=1,2  \\[8pt]
\displaystyle
(1+\mbox{e}^{ms})(\epsilon+\mbox{e}^s) \nabla z+\epsilon z=-L(\mbox{e}^s)(\mbox{e}^s-\theta_0) 
\end{array}\right\} \mbox{ at } \partial \Omega,
\end{array}
\end{equation}
where $\rho={\cal S}(\bw)$ is given by Lemma \ref{l2}.

Our aim is to apply the Leray--Schauder fixed point theorem. Thus we need to verify that ${\cal T}$  
is a continuous and compact mapping from $M_p \times W^2_p(\Omega)$ to $M_p \times W^2_p(\Omega)$
and that all solutions satisfying 
\begin{equation}\label{c3a}
t {\cal T}(\bw,z)=(\bw,z), \qquad t \in [0,1]  \mbox{ \ \ \ are bounded in $M_p \times W^2_p(\Omega)$.}
\end{equation}

First we easily have

\begin{lemma} \label{l3}
Let $p>3$ and all assumptions of Theorem \ref{t2} be satisfied. Then  ${\cal T}$ is a continuous and compact operator from $M_p \times W^2_p(\Omega)$ to $M_p \times W^2_p(\Omega)$.
\end{lemma}

\begin{proof}
Note that for $\epsilon >0$ the system (\ref{L-S}) is strictly elliptic. Since $p>3$,  the $W^1_p(\Omega)$--space is algebra, thus the r.h.s. of (\ref{L-S}) belongs to the $L_p$--space 
(the boundary term belongs to $W^{1-1/p}_{p}(\partial \Omega)$).
The coefficients in the operator in the l.h.s. of $(\ref{L-S})_2$ 
are of the $C^{1+\alpha}(\overline\Omega)$--class. Hence the standard theory for elliptic
systems gives us the existence of the solution to (\ref{L-S}) in $M_p\times W^2_p(\Omega)$ with 
the following bound
%
$$
\begin{array}{c}
||\bw||_{W^2_p(\Omega)}+||z||_{W^2_p(\Omega)}\leq C(||e^s||_{C^{1+\alpha}(\overline\Omega)})\Big(||\mbox{the r.h.s. of }(\ref{L-S})_1||_{L_p(\Omega)}\\[8pt]
+||\mbox{the r.h.s. of }(\ref{L-S})_2||_{L_p(\Omega)}
+||\mbox{the r.h.s. of }(\ref{L-S})_4||_{W^{1-1/p}_p(\partial \Omega)}\Big)
\end{array}
$$
%
which
guarantees us the uniqueness and the continuous dependence on the data.
Moreover the r.h.s. of (\ref{L-S}) is at most of the first order of sought functions. Thus this structure implies the compactness for the 
map ${\cal T}$.
\end{proof}

Next we consider a priori bounds for solutions to (\ref{c3a}).

\begin{lemma} \label{l4} 
All solutions to problem (\ref{c3a}) in the class $M_p \times W^2_p(\Omega)$  satisfy the following bounds
\begin{equation}\label{L1}
0\leq \rho \leq k, \qquad 
||\bw||_{H^1(\Omega)}+||\theta||_{L_{3m}(\Omega)}+ ||\nabla \theta||_{L_r(\Omega)}+\sqrt{\varepsilon}\|\nabla \rho\|_{L_2(\Omega)}\leq C(k),
\end{equation}
where $r= \min\{\frac{3m}{m+1},2\}$, $\theta=e^z$ and the constant C(k) is independent of $\epsilon$ and $t \in [0,1]$.
\end{lemma}

\begin{proof}
We may basically repeat estimates of Lemma \ref{ene-bound} from the previous section. However, on the one hand, we are in a simpler situation as we can use bounds which depend on $k$, i.e. on the $L_\infty$ bound of the density (they may be proved analogously as in \cite{MuPo}), on the other hand, we must control the behavior of all norms with respect to $t$.

Thus, repeating steps (\ref{b8a})--(\ref{b3}) for the case $t=1$ (the corresponding terms are only multiplied by $t$) we finally get
\begin{equation*}\label{L1a}
\begin{array}{c}
\displaystyle
(1-t) \int_\Omega {\bf S}(\bw):\nabla \bw dx + \int_{\partial \Omega} f (\bw \odot \btau)^2 d\sigma + \int_\Omega
\frac{(1+\theta^m)(\epsilon +\theta)}{\theta} |\nabla z|^2 dx \\[8pt]
\displaystyle +
t \int_\Omega \Big( \frac{{\bf S}(\bw):\nabla \bw}{\theta} +
\epsilon \gamma \rho^{\gamma-2}|\nabla \rho|^2  +  \frac{\epsilon\gamma}{\gamma-1}\rho^\gamma \Big)dx \\[8pt]
\displaystyle
+\epsilon \int_{\partial \Omega} \big[ z_+(1-\mbox{e}^{-z_+})+|z_-|(\mbox{e}^{|z_-|}-1)\big]d \sigma
+t \int_{\partial \Omega} \big[L(\theta)\theta-L(\theta)\theta_0  +\frac{L(\theta)\theta_0}{\theta}- L(\theta)
\big] d\sigma
\\[8pt]
\displaystyle
\leq t \int_\Omega \Big(K(\rho)\rho \bw \cdot\nabla z -K(\rho)\bw \cdot \nabla \rho\Big)  dx  +
tC\Big(1+\int_\Omega |K(\rho)\rho \bw \cdot \bF| dx\Big),
\end{array}
\end{equation*}
where $\rho = S(\bw)$. 

We may now repeat the arguments between (\ref{2.16})--(\ref{b4}) (all the corresponding terms are only multiplied by $t$) and we finally get

\begin{equation*}\label{L1e}
\begin{array}{c}
\displaystyle
\int_\Omega \frac{1+\theta^m}{\theta^2}|\nabla \theta|^2 dx+
t\int_\Omega  \frac{{\bf S}(\bw):\nabla \bw}{\theta} dx
+\int_{\partial \Omega} \Big(tL(\theta) \theta +t\frac{L(\theta)\theta_0}{\theta}
+\epsilon |z|\Big) d \sigma
\\[8pt]
\displaystyle
\leq tC\Big(1+\int_\Omega |K(\rho)\rho \bv \bF| dx\Big).
\end{array}
\end{equation*}

As $0\leq \rho\leq k$, we easily get (the Poincar\'e inequality is just the same as in the previous section), after dividing by $t$ (the case $t=0$ is clear; recall also $m=l+1$)
\begin{equation*}\label{L1f}
||\theta||_{L_{3m}(\Omega)}\leq C(1+||\bw||_{L_2(\Omega)})^{1/m}
\end{equation*}
and from an analogue to (\ref{b10}) also
\begin{equation*}\label{L1g}
||\bw||_{H^1(\Omega)}^2 \leq C(1+||\theta||^2_{L_2(\Omega)}).
\end{equation*}
As $m>1$, it implies
$$
||\bw||_{H^1(\Omega)}+||\theta||_{L_{3m}(\Omega)} \leq C(k).
$$
Further, if $m\geq2$ then due to the control of $\frac{|\nabla\theta|}{\theta}$ and $|\nabla \theta| \theta^{\frac{m-2}{2}}$ in 
$L_2(\Omega)$ we have  also $\nabla \theta$  bounded in the same space. For $1<m<2$,
\begin{equation} \label{3.13}
\|\nabla \theta\|_{L_{\frac{3m}{m+1}}(\Omega)} \leq \||\nabla \theta| \theta^{\frac{m-2}{2}}\|_{L_2(\Omega)} \|\theta\|_{L_{3m}(\Omega)}^{\frac{2-m}{2}}.
\end{equation}
Finally, multiplying the approximative continuity equation by $\rho$ and integrating by parts we get  %
\begin{equation*}\label{L3}
\epsilon \int_\Omega (|\nabla \rho |^2+ \rho^2) dx\leq  \epsilon \int_\Omega 
hK(\rho)\rho dx + \int_\Omega \Big(\int_0^\rho K(t)t dt\Big) |\div \bw| dx,
\end{equation*}
from where we deduce the bound for $\sqrt{\epsilon} \|\nabla \rho\|_{L_2(\Omega)}$. 
\end{proof}
 
To conclude, we  verify the bound on $(\bw,z)$ in $W^2_p(\Omega) \times W^2_p(\Omega)$, $p < \infty$, independently of $t$.  
We apply the bootstrap method to system 
\begin{equation}\label{L-S-t}
\begin{array}{c}
\left.
\begin{array}{c}
\displaystyle
-\div {\bf S}(\bw)=t\Big[-\frac 12 \div(K(\rho)\rho \bw \otimes \bw)
-\frac 12 K(\rho)\rho \bw \cdot \nabla \bw  \\[4pt]
\displaystyle -\nabla P(\rho,\mbox{e}^z) + K(\rho)\rho \bF\Big] \\[10pt]
\displaystyle
-\div\big((1+\mbox{e}^{mz})(\epsilon+\mbox{e}^z)\nabla z\big)=
t\Big[{\bf S}(\bw):\nabla \bw -  
\div \Big(\bw \int_0^\rho K(t)dt\Big)\mbox{e}^z\\[4pt]
\displaystyle -\div\big(K(\rho)\rho \bw\big)\mbox{e}^z-
\mbox{e}^z K(\rho)\rho \bw \cdot \nabla z+ \mbox{e}^z K(\rho) \bw \cdot \nabla \rho \Big] 
\end{array}
\right\} \mbox{ in } \Omega,
\\[60pt]
\left.
\begin{array}{c}
 \displaystyle
\bw\cdot \bn=0,\;\;\;
 \bn \cdot {\bf S}(\bw)\cdot \btau_l+f\bw \cdot \btau_l=0 \mbox{ \ \ for } l=1,2   \\[8pt]
\displaystyle
(1+\mbox{e}^{mz})(\epsilon+\mbox{e}^z) \nabla z+\epsilon z=-tL(\mbox{e}^z)(\mbox{e}^z-\theta_0)
\end{array}
\right\} 
\mbox{ at } \partial \Omega, 
\end{array}
\end{equation}
where $\rho={\cal S}(\bw)$ given by Lemma \ref{l2}. Note first that due to bounds from Lemma \ref{l4} we have
$$
\|\bw\|_{W^1_3(\Omega)} \leq C
$$
as $K(\rho) \rho \bw\otimes \bw$ is bounded in $L_3(\Omega)$. Thus $\bw$ is bounded in any $L_q(\Omega)$, $q<\infty$ and the most restrictive term is $\nabla P(\rho,\mbox{e}^z)$. As $\mbox{e}^z = \theta$ is bounded in $L_{3m}(\Omega)$, $\rho$ in $L_\infty(\Omega)$, we deduce the bound
$$
\|\bw\|_{W^1_{3m}(\Omega)} \leq C
\mbox{ \ \ \ and consequently also   \ \ }
\|\rho\|_{W^2_{3m}(\Omega)} \leq C.
$$
Note that the constant in the estimate for $\bw$ is independent of $\epsilon$.

Next, we rewrite equation (\ref{L-S-t})$_2$ as follows
\begin{equation}\label{heat}
\begin{array}{c}
\displaystyle
-\Delta \Phi(z)=t
\Big[{\bf S}(\bw):\nabla \bw +\mbox{e}^z K(\rho)\rho\bw \cdot\nabla z -\mbox{e}^z K(\rho)\bw\cdot \nabla\rho  \\[4pt]
\displaystyle
-\div \Big(\bw\int_0^\rho K(t)dt\Big)\mbox{e}^z-\div\big(K(\rho)\rho \bw\big)\mbox{e}^z \Big]
 \mbox{ \ \ in \ }  \Omega,\\[13pt]
\displaystyle
\frac{\partial \Phi(z)}{\partial \bn}=-\epsilon z -tL(\mbox{e}^z)(\mbox{e}^z-\theta_0) 
\mbox{ \ \ at \ } \partial \Omega
\end{array}
\end{equation}
with 
\begin{equation}\label{D1}
\Phi(z)=\int_0^x (1+\mbox{e}^{m\tau})(\epsilon+\mbox{e}^\tau) d \tau.
\end{equation}
We multiply (\ref{heat})$_1$ by $\Phi$ and integrate over $\Omega$. It leads to
\begin{equation*}\label{D2}
||\nabla \Phi||_{L_2(\Omega)}^2+\int_{\partial \Omega} \big(tL(\mbox{e}^z)(\mbox{e}^z-\theta_0) \Phi + \epsilon z \Phi\big)d\sigma
\leq C||\mbox{the r.h.s. of }(\ref{heat})_1||_{L_{6/5}(\Omega)}||\Phi||_{L_6(\Omega)}.
\end{equation*}
It is not difficult to realize that the most restrictive term on the r.h.s is $\mbox{e}^z K(\rho)\rho\bw \cdot\nabla z\in L_{\frac{3m}{m+1}}(\Omega)$, where $\frac{3m}{m+1} > \frac 65$ for $m>1$.

Let us look at the boundary terms. Note that $\Phi(s) \sim \epsilon s$ for $s\to -\infty$ and $\Phi \sim \mbox{e}^{(m+1)s}$ for $s\to +\infty$. Thus 
\begin{equation*}\label{D5a}
\int_{\partial \Omega} \big[tL(\mbox{e}^s)(\mbox{e}^s-\theta_0)\Phi + \epsilon s \Phi\big]I_{\{\Phi \leq 0\}} d\sigma \geq ||\Phi||_{L_2(\partial \Omega)}-C
\end{equation*}
and
\begin{equation*}\label{D5b}
\int_{\partial \Omega} \big[tL(\mbox{e}^s)(\mbox{e}^s-\theta_0)\Phi + \epsilon s \Phi\big]I_{\{\Phi \geq 0\}} d\sigma \geq ||\Phi||_{L_1(\partial \Omega)}-C.
\end{equation*}
Thus, the estimates above yield
$\|\Phi\|_{W^1_2(\Omega)} \leq C$
with $C$ independent of $t$ which implies
$$
\|\theta^{m+1}\|_{L_6(\Omega)} = \|\mbox{e}^{(m+1)z}\|_{L_6(\Omega)} \leq C
\mbox{ \ \ and also \ \ }
\|\nabla \theta\|_{L_2(\Omega)} = \|\mbox{e}^{z}\nabla z\|_{L_2(\Omega)} \leq C.
$$ 
Now, it is not difficult to verify that from
(\ref{heat}) we get
$\|\Phi\|_{W^2_{p^*}(\Omega)} \leq C$
with $p^*=\min\{\frac{3m}{2},2\}$ (thus $\mbox{e}^z \nabla z \in L_2(\Omega)$ and $\nabla w \in L_{3m}(\Omega)$).
In particular, 
$$
\begin{array}{c}
\|z\|_{L_\infty(\Omega)} + \|\theta\|_{L_\infty(\Omega)} \leq C ,\qquad
\|\nabla z\|_{L_q(\Omega)} + \|\nabla \theta\|_{L_q(\Omega)} \leq C
\end{array}
$$
for $1\leq q\leq q^* = \frac{3p^*}{3-p^*} >3$. Thus from the approximative momentum equation we get 
($\nabla (\rho \theta) \in L_{q^*}(\Omega)$) the bound
$\|\bw \|_{W^2_{q^*}(\Omega)} \leq C$
and from the energy/entropy equation also
$$
\|z\|_{W^2_{q^*}(\Omega)} + \| \theta\|_{W^2_{q^*}(\Omega)} \leq C.
$$
The imbedding theorem yields
$\|\nabla z\|_{L_\infty(\Omega)} + \|\nabla \theta\|_{L_\infty(\Omega)} \leq C$
which finally gives as above
$$
\|\bw\|_{W^2_r(\Omega)} + \|z\|_{W^2_{r}(\Omega)} + \| \theta\|_{W^2_{r}(\Omega)} \leq C, \quad 1\leq r <\infty
$$
with $C$ independent of $t$. This finishes the proof of Theorem \ref{t2}.

\section{Effective viscous flux}

In this part we investigate the properties of the effective viscous flux.
Estimates (\ref{b-m}) from Theorem \ref{t2} guarantee us existence of a subsequence $ \epsilon \to 0^+$ such that
\begin{equation}\label{e2}
\begin{array}{c}
\bv_\epsilon \rightharpoonup \bv \mbox{ \ \ \ in } W^1_{3m}(\Omega),\\
\bv_\epsilon \to \bv \mbox{ \ \ \ in } L_\infty(\Omega),\\
\rho_\epsilon \rightharpoonup^* \rho \mbox{ \ \ \ in } L_\infty(\Omega),\\
P_b(\rho_\epsilon) \rightharpoonup^* \overline{P_b(\rho)} \mbox{ \ \ \ in } L_\infty(\Omega),\\
K(\rho_\epsilon)\rho_\epsilon \rightharpoonup^* \overline{K(\rho)\rho}\mbox{ \ \ \ in } L_\infty(\Omega),\\
K(\rho_\epsilon)\rightharpoonup^* \overline{K(\rho)}\mbox{ \ \ \ in } L_\infty(\Omega),\\
\displaystyle \int_0^{\rho_\epsilon} K(t)dt \rightharpoonup^* \overline{\int_0^\rho K(t)dt}\mbox{ \ \ \ in } L_\infty(\Omega),\\
\theta_\epsilon \rightharpoonup \theta \mbox{ \ \ \ in } W^1_r(\Omega) \mbox{ with }
r=\min\{2,\frac{3m}{m+1}\},\\
\theta_\epsilon \to \theta \mbox{ \ \ \ \ in }L_q(\Omega) \mbox{\ \ \ for } q< 3m.
\end{array}
\end{equation}
Passing to the limit in the weak formulation of our problem we get

\begin{equation}\label{4.1}
\div (\overline{K(\rho)\rho} \bv) =0,
\end{equation} 
\begin{equation}\label{4.2}
\overline{K(\rho)\rho} \bv \cdot \nabla \bv - \div\Big(2\mu \bD(\bv)+\nu (\div \bv) \bI-
\overline{P_b(\rho)}\bI-\theta \big(\overline{\int_0^\rho K(t)dt}\big)\bI \Big)=\overline{K(\rho)\rho}\bF,
\end{equation}
\begin{equation} \label{4.3}
-\div((1+\theta^m)\nabla \theta) + \theta \big(\overline{\div \bv \int_0^\rho K(t) dt }\big) + \div (\overline{K(\rho)\rho }\theta \bv) = 
2\mu\overline{|D(\bv)|^2} + \nu \overline{(\div \bv)^2}
\end{equation}
together with the boundary conditions (\ref{1.8})--(\ref{1.10}). Recall that (\ref{4.1})--(\ref{4.3}) is satisfied in the weak sense,
similar to Definition \ref{d1}.

In what follows we must carefully study the dependence of the a priori bounds on $k$. We have

\begin{lemma} \label{l 4.1} 
Under the assumptions of Theorems \ref{t1} and \ref{t2}, we have 
\begin{equation}\label{ee1}
||\rho_\epsilon||_{L_\infty(\Omega)} \leq k \mbox{ \ \ and \ \  }
||\bv_\epsilon||_{W^1_{3m}(\Omega)}\leq C(1+k^{\frac \gamma 3 \frac{3m-2}{m}}).
\end{equation}
\end{lemma}
\smallskip

\begin{proof} The bound on the density follows directly from Theorem \ref{t2}.  We therefore estimate the velocity. If we write
(\ref{ap-ns})$_2$ in the form
\begin{equation*}\label{ee2}
\begin{array}{c}
\displaystyle -\div \bS(\bv)=-\nabla \Big(P_b(\rho_\epsilon) 
+\theta_\epsilon \big(\int_0^{\rho_\epsilon} K(t)dt \big)\Big) +K(\rho_\epsilon)\rho_\epsilon
\bF\\ 
\displaystyle -\frac{1}{2} \div [K(\rho_\epsilon)\rho_\epsilon \bv_\epsilon \otimes \bv_\epsilon] 
-\frac{1}{2} K(\rho_\epsilon)\rho_\epsilon \bv_\epsilon \cdot \nabla \bv_\epsilon,
\end{array}
\end{equation*}
we immediately see that
$$
\begin{array}{c}
\displaystyle \|\bv_\epsilon\|_{W^1_{3m}(\Omega)}\leq C \big(\|K(\rho_\epsilon)\rho_\epsilon \bv_\epsilon \otimes \bv_\epsilon\|_{L_{3m}(\Omega)} + 
\|K(\rho_\epsilon)\rho_\epsilon \bv_\epsilon \cdot \nabla \bv_\epsilon\|_{L_{\frac{3m}{m+1}}(\Omega)}  \\
\displaystyle +\|P_b(\rho_\epsilon)\|_{L_{3m}(\Omega)} + 
\|\theta_\epsilon \big(\int_0^{\rho_\epsilon} K(t)dt\big)\|_{L_{3m}(\Omega)} + \|K(\rho_\epsilon)\rho_\epsilon
\bF \|_{L_{\frac{3m}{m+1}}(\Omega)}\big).
\end{array}
$$
Note that due to the bound of the temperature we cannot expect $\epsilon$--independent estimate for $q>3m$. The bounds on the density
and temperature yield
$$
\|P_b(\rho_\epsilon)\|_{L_{3m}(\Omega)}  \leq \|P_b(\rho_\epsilon)\|^{\frac{2}{3m}}_{L_{2}(\Omega)} \|P_b(\rho_\epsilon)\|^
{\frac{3m-2}{3m}}_{L_{\infty}(\Omega)}\leq C k^{\gamma \frac{3m-2}{3m}},
$$
while
$$
\|\theta_\epsilon \big(\int_0^{\rho_\epsilon} K(t)dt\big)\|_{L_{3m}(\Omega)} \leq C k.
$$
Note that for $m$ and $\gamma$ satisfying assumptions of Theorem \ref{t1}, $\gamma \frac{3m-2}{3m} >1$. It remains to estimate the
convective terms ($C.T.$)
$$
\begin{array}{c}
\displaystyle
C.T. \leq \|K(\rho_\epsilon)\rho_\epsilon |\bv_\epsilon|^2 \|_{L_{3m}(\Omega)} + 
\|K(\rho_\epsilon)\rho_\epsilon |\bv_\epsilon| |\nabla \bv_\epsilon|\|_{L_{\frac{3m}{m+1}}(\Omega)}\\[8pt]
\displaystyle  \leq C \|\rho_\epsilon\|_{L_\infty(\Omega)}
\big(\|\bv_\epsilon\|_{L_{6m}(\Omega)}^2 + \|\nabla \bv_\epsilon\|_{L_{\frac{3m}{m+1}}(\Omega)}\|\bv_\epsilon\|_{L_\infty(\Omega)}\big)
\end{array}
$$
for $m\geq 2$, while for $m<2$ the last term is replaced by $\|\nabla \bv_\epsilon\|_{L_2(\Omega)}\|\bv_\epsilon\|_{L_{\frac{6m}{2-m}}(\Omega)}$. Using the
fact that for $6<q\leq \infty$
$$
\|\bv_\epsilon\|_{L_q(\Omega)} \leq C\|\bv_\epsilon\|_{L_6(\Omega)}^\alpha \|\bv_\epsilon\|_{W^1_{3m}(\Omega)}^{1-\alpha}
\mbox{ \ \ with \ \ } \qquad \frac{1}{q}= \frac{\alpha}{6} +
(1-\alpha)\big(\frac{1}{3m}-\frac 13\big)
$$
and for $2<r<3m$
$$
\|\nabla \bv_\epsilon\|_{L_r(\Omega)} \leq \|\bv_\epsilon\|_{L_2(\Omega)}^\alpha \|\nabla\bv_\epsilon\|_{L_{3m}(\Omega)}^{1-\alpha}
\mbox{ \ \ with \ \ } \qquad \frac{1}{r}= 
\frac{\alpha}{2} +\frac{1-\alpha}{3m},
$$   
we end up with
$$
C.T. \leq C \|\rho_\epsilon\|_{L_\infty(\Omega)}
\|\bv_\epsilon\|_{W^1_{2}(\Omega)}^{2\frac{2m-1}{3m-2}} \|\bv_\epsilon\|_{W^1_{3m}(\Omega)}^{2\frac{m-1}{3m-2}}.
$$
Note that $\frac{2(m-1)}{3m-2}<1$. Thus we may use the bound on $\rho_\epsilon$ and Young's inequality yields
$$
\|\bv_\epsilon\|_{W^1_{3m}(\Omega)} \leq C(1+k^{\frac{\gamma}{3} \frac{3m-2}{m}}) + C k^{\frac{3m-2}{m}} + 
\frac 12 \|\bv_\epsilon\|_{W^1_{3m}(\Omega)}.
$$
As $\gamma>3$, the lemma is proved.
\end{proof}

Before using the above proved bounds, we show one useful result which in particular implies that the limit temperature is positive.

\begin{lemma} \label{l 4.2} There exists a subsequence  $\{s_\epsilon\}$ such that
$$
s_\epsilon \to s \mbox{ in }L_2(\Omega),
$$
subsequently,
$$
\theta_\epsilon \to \theta \mbox{ in } L_q(\Omega), \quad q<3m  \mbox{ \ \ with \ \ } \theta >0 \quad \mbox{ a.e. in } \Omega.
$$ 
\end{lemma}

\begin{proof}
Recall that from the energy bound we have the following information
\begin{equation*}\label{h8}
\int_\Omega |\nabla s_\epsilon|^2dx + \int_{\partial \Omega} (e^{s_\epsilon}+e^{-s_\epsilon}) d \sigma < C
\end{equation*}
which in particular gives
\begin{equation*}\label{h9}
\int_\Omega |\nabla s_\epsilon|^2dx + \int_{\partial \Omega} s^2_\epsilon d \sigma < C.
\end{equation*}
Thus we are allowed to choose a subsequence $s_\epsilon \to s$ in $L_2(\Omega)$. Recall also that 
$\theta_\epsilon = e^{s_\epsilon}$ and $\theta_\epsilon \to \theta$ strongly in
$L_r(\Omega)$, $r<3m$. Hence by Vitali's theorem (for a subsequence, if necessary)
$$
\begin{array}{c}
\displaystyle 
e^{s_\epsilon} \to e^s \qquad \mbox{ in } L_r(\Omega) \qquad \mbox{ and }
 \theta = e^s \qquad \mbox{ with } s \in L_2(\Omega).
\end{array}
$$
Thus $\theta >0$ a.e. in $\Omega$ as $s>-\infty$ a.e. in $\Omega$.
\end{proof}

A crucial role in the proof of the strong convergence of the density is played by a quantity called the effective viscous flux. To define
it, we need the Helmholtz decomposition of the velocity
\begin{equation}\label{e4}
\bv=\nabla \phi  +\rot \, \bA,
\end{equation}
where the divergence-free part of the velocity is given as a solution to the following 
elliptic problem
\begin{equation}\label{e5}
\begin{array}{c}
\rot\,\rot\,\bA=\rot\,\bv=\bomega \mbox{ \ \ \ \ in } \Omega,\\
\div\,\rot\, \bA=0 \mbox{ \ \ \ \ in } \Omega,\\
\rot \bA\cdot \bn =0 \mbox{ \ \ \ \ at } \partial \Omega.
\end{array}
\end{equation}
The potential part of the velocity is given  by
the solution to 
\begin{equation}\label{e6}
\begin{array}{c}
\Delta \phi=\div \bv \mbox{ \ \ in } \Omega,\\
\frac{\partial \phi}{\partial \bv}=0 \mbox{ \ \ at } \partial \Omega,
\end{array}
\qquad\qquad  \int_\Omega \phi dx =0.
\end{equation}
The classical theory for elliptic equations gives us for $1<q<\infty$
\begin{equation*}\label{e7}
\begin{array}{rcl}
||\nabla \rot\,\bA||_{L_q(\Omega)}\leq C||\bomega||_{L_q(\Omega)},
& \qquad &
||\nabla^2 \rot\,\bA||_{L_q(\Omega)}|| \leq C||\bomega||_{W^1_q(\Omega)}, \\[8pt]
||\nabla^2 \phi||_{L_q(\Omega)}\leq C||\div \bv||_{L_q(\Omega)}, &\qquad & 
||\nabla^3 \phi||_{L_q(\Omega)}\leq C||\div \bv||_{W^1_q(\Omega)}.
\end{array}
\end{equation*}

The properties of the slip boundary condition enables us to state the following problem
\begin{equation}\label{e8}
\begin{array}{c}
\displaystyle -\mu\Delta \bomega_\epsilon= \rot\big(K(\rho_\epsilon)\rho_\epsilon \bF - K(\rho_\epsilon)\rho_\epsilon
\bv_\epsilon \cdot \nabla \bv_\epsilon \\[8pt]
\displaystyle -
\frac{1}{2} \epsilon h K(\rho_\epsilon)\bv_\epsilon + \frac{1}{2} \epsilon \rho_\epsilon \bv_\epsilon\big)
-\rot(\frac{1}{2} \epsilon \Delta \rho_\epsilon \bv_\epsilon):=\bH_1+\bH_2 \mbox{ \ \ in } \Omega, \\[12pt]
\bomega_\epsilon\cdot \btau_1=-(2\chi_2 -f/\mu)\bv_\epsilon \cdot \btau_2 \mbox{ \ \ at } \partial \Omega,\\[8pt]
\bomega_\epsilon\cdot \btau_2=(2\chi_1 -f/\mu)\bv_\epsilon \cdot \btau_1 \mbox{ \ \ at } \partial \Omega,\\[8pt]
\div \bomega_\epsilon =0 \mbox{ \ \ at }\partial \Omega,
\end{array}
\end{equation}
where $\chi_k$ are curvatures related with directions $\btau_k$. For the proof of relations $(\ref{e8})_{2,3}$ -- see 
\cite{Mu} or \cite{MuRa}.

The structure of $\bomega_\epsilon$ gives us a hint to consider it as a sum of three components
\begin{equation}\label{e9}
\bomega_\epsilon=\bomega_\epsilon^0+\bomega_\epsilon^1+\bomega_\epsilon^2,
\end{equation}
where they are determined by the following systems
\begin{equation}\label{e10}
\begin{array}{cccr}
-\mu \Delta \bomega_\epsilon^0=0, &-\mu \Delta \bomega_\epsilon^1=\bH_1, &
-\mu \Delta \bomega_\epsilon^2=\bH_2 & \mbox{ in } \Omega,\\
\bomega_\epsilon^0\cdot \btau_1=-(2\chi_2 -f/\mu)\bv_\epsilon \cdot \btau_2, &
\bomega_\epsilon^1\cdot \btau_1=0, & \bomega_\epsilon^2\cdot \btau_1=0 & \mbox{ at } \partial \Omega,\\
\bomega_\epsilon^0\cdot \btau_2=(2\chi_1 -f/\mu)\bv_\epsilon \cdot \btau_1, &
\bomega_\epsilon^1\cdot \btau_2=0, & \bomega_\epsilon^2\cdot \btau_2=0 & \mbox{ at } \partial \Omega,\\
\div \bomega_\epsilon^0=0, & \div \bomega_\epsilon^1=0, &\div \bomega_\epsilon^2=0 & \mbox{ at } \partial \Omega.
\end{array}
\end{equation}
\begin{lemma} \label{l 4.3} For the vorticity $\bomega_\epsilon$ written in the form (\ref{e9}) we have:\footnote{Note that we can prove
that $\|\bomega_\epsilon^2\|_{L_r(\Omega)} = o(\epsilon)$ for $\epsilon \to 0^+$ for any $r<3m$. As we do not need it and the proof of the rate is
slightly more complicated, we skip it. Analogously we may consider the other inequality also for $q<2$, with different powers of $k$.}
\begin{equation}\label{e11}
\begin{array}{c}
||\bomega_\epsilon^2||_{L_r(\Omega)}\leq C(k)\epsilon^{1/2} \mbox{ \ \ \ for } 1\leq r \leq 2,
\\[8pt]
||\bomega_\epsilon^0||_{W^1_q(\Omega)}+||\bomega_\epsilon^1||_{W^1_q(\Omega)}\leq C(1+k^{1+\gamma (\frac 43 -\frac 2q)}) \mbox{ \ \ \ for }
2\leq q \leq 3m.
\end{array}
\end{equation}
\end{lemma}

\begin{proof} First, let us consider $\bomega_\epsilon^0$. Take $\balpha_0$ any divergence--free extension of the boundary data to
$\bomega_\epsilon$, e.g. in the form of a solution to the following Stokes problem
\begin{equation}\label{e12}
\begin{array}{c}
-\mu\Delta \balpha_0 +\nabla p_0=0 \mbox{ \ \ \ \ \ \ in } \Omega,\\
\div \balpha_0=0 \mbox{ \ \ \ \ \ \  in } \Omega,\\
\balpha_0\cdot \btau_1=-(2\chi_2 -f/\mu)\bv_\epsilon \cdot \btau_2  \mbox{ \ \ \ \ \ \ at } 
\partial \Omega,\\
\balpha_0\cdot \btau_2=(2\chi_1 -f/\mu)\bv_\epsilon \cdot \btau_2
\mbox{ \ \ \ \ \ \  at } \partial \Omega,\\
\balpha_0\cdot \bn=0 \mbox{ \ \ \ \ \ \ at } \partial \Omega.
\end{array}
\end{equation}
Note that $\bv_\epsilon \in W^{1-1/(3m)}_{3m}(\partial \Omega)$, thus 
$\balpha_0 \in W^1_{3m}(\Omega)$ with the estimate
$$
\|\balpha_0\|_{W^1_q(\Omega)}\leq C \|\bv_\epsilon\|_{W^1_q(\Omega)}, \qquad 1<q\leq3m.
$$
Thus we may transform the system for $\bomega_\epsilon^0$ to the form
\begin{equation}\label{e13}
\begin{array}{c}
-\mu\Delta (\bomega_\epsilon^0-\balpha_0)=\mu\Delta \balpha_0 \mbox{ \ \ \ in } \Omega,\\
(\bomega_\epsilon^0-\balpha_0)\cdot \btau_1=0 \mbox{ \ \ \ at } \partial \Omega,\\
(\bomega_\epsilon^0-\balpha_0)\cdot \btau_2=0 \mbox{ \ \ \ at } \partial \Omega,\\
\div (\bomega_\epsilon^0 -\balpha_0)=0 \mbox{ \ \ \ at } \partial \Omega.
\end{array}
\end{equation}
Note that $\Delta \balpha_0 \in W^{-1}_{3m}(\Omega)$. Here $W^{-1}_p(\Omega)$ denotes the dual space to 
$$
\{ \bff \in W^1_p(\Omega)\cap \{ \bff\cdot \btau_1= \bff \cdot \btau_2=0 \mbox{ at } \partial \Omega\} \}.
$$

As the system for $\bomega_\epsilon^0$ has the same structure as that for $\bomega_\epsilon^1$, we get (see \cite{So}, \cite{Zaj}):
\begin{equation*}\label{e14}
||\bomega_\epsilon^1||_{W^1_q(\Omega)}\leq C||\bH_1||_{W^{-1}_{q}(\Omega)}
\mbox{ \ \ and \ \ }
||\bomega_\epsilon^0||_{W^1_q(\Omega)}\leq C||\bv_\epsilon||_{W^1_q(\Omega)}, \quad 1<q\leq 3m.
\end{equation*}

Analyzing the form of $\bH_1$ we see that the only not elementary term is the convective one; so
we obtain
\begin{equation*}\label{e15}
||\bomega_\epsilon^1||_{W^1_q(\Omega)}\leq C(1+||K(\rho_\epsilon)\rho_\epsilon \bv_\epsilon \cdot \nabla
\bv_\epsilon||_{L_q(\Omega)}).
\end{equation*}

We easily see that for $q\geq 2$
\begin{equation*}\label{e16}
||K(\rho_\epsilon)\rho_\epsilon \bv_\epsilon \cdot \nabla \bv_\epsilon||_{L_q(\Omega)}\leq
k||\bv_\epsilon||_{L_\infty(\Omega)}||\nabla \bv_\epsilon||_{L_q(\Omega)}.
\end{equation*}
Using interpolation inequalities as in Lemma \ref{l 4.1} we prove that
\begin{equation*}\label{e19}
\begin{array}{rcl}
\displaystyle 
||K(\rho_\epsilon)\rho_\epsilon \bv_\epsilon \cdot \nabla \bv_\epsilon||_{L_q(\Omega)} &\leq& Ck 
\|\bv_\epsilon\|_{L_6(\Omega)}^{\frac{2(m-1)}{3m-2}} 
\|\nabla \bv_\epsilon\|_{L_{3m}(\Omega)}^{\frac{m}{3m-2}} \|\nabla \bv_\epsilon\|_{L_2(\Omega)}^{\frac{6m-2q}{(3m-2)q}} 
\|\nabla \bv_\epsilon\|_{L_{3m}(\Omega)}^{\frac{3m(q-2)}{(3m-2)q}} \\[8pt]
\displaystyle &\leq& C k ^{1+ \gamma(\frac 43-\frac 2q)}.
\end{array}
\end{equation*}
Evidently, the estimate for $\bomega_\epsilon^0$ is less restrictive.

Similarly, for $\bomega_\epsilon^2$ we have
\begin{equation*}\label{e21}
||\bomega^2_\epsilon||_{L_q(\Omega)} \leq C||\epsilon \Delta \rho_\epsilon \bv_\epsilon||_{W^{-1}_q(\Omega)}\leq
C\epsilon\sup_{\phi}|\int_\Omega \Delta \rho_\epsilon \bv_\epsilon \phi dx|,
\end{equation*}
where the sup is taken over all functions belonging to $W^1_q(\Omega)$ with $1/p+1/q=1$.

From the continuity equation we know that
\begin{equation*}\label{e22}
\sqrt{\epsilon}||\nabla \rho_\epsilon||_{L_2(\Omega)} \leq C(k).
\end{equation*}
(For $q>2$ we have only $\epsilon \|\nabla \varrho_\epsilon\|_{L_q(\Omega)} \leq C.$) As $q\leq 2$,

\begin{equation*}\label{e23}
||\bomega^2_\epsilon||_{L_q(\Omega)}\leq C\epsilon (\|\nabla \rho_\epsilon\|_{L_2(\Omega)} \|\bv_\epsilon\|_{L_\infty(\Omega)} + 
\|\nabla \rho_\epsilon\|_{L_2(\Omega)} \|\nabla \bv_\epsilon\|_{L_{3m}(\Omega)}) \leq C(k) \epsilon^{\frac 12}.
\end{equation*}
The lemma is proved.
\end{proof}

\smallskip

We now introduce the fundamental quantity --- the effective viscous flux --- which is in fact the potential part of the momentum equation.
Using the Helmholtz decomposition in the approximative momentum equation we have 
\begin{equation*}\label{e24}
\begin{array}{c}
\displaystyle \nabla (-(2 \mu+\nu) \Delta \phi_\epsilon +P(\rho_\epsilon,\theta_\epsilon))=
\mu \Delta \rot \bA_\epsilon + K(\rho_\epsilon)\rho_\epsilon \bF 
\\[12pt]
\displaystyle - K(\rho_\epsilon)\rho_\epsilon
\bv_\epsilon \cdot \nabla \bv_\epsilon 
-\frac 12 \epsilon h K(\rho_\epsilon)\bv_\epsilon +\frac 12 \epsilon \rho_\epsilon \bv_\epsilon
-\frac 12 \epsilon\Delta \rho_\epsilon \bv_\epsilon.
\end{array}
\end{equation*}
We define
\begin{equation} \label{evf_eps}
G_\varepsilon = -(2 \mu+\nu) \Delta \phi_\epsilon +P(\rho_\epsilon,\theta_\epsilon) = -(2 \mu+\nu) \div \bv_\epsilon 
+P(\rho_\epsilon,\theta_\epsilon)
\end{equation}
and its limit version
\begin{equation} \label{evf}
G =  -(2 \mu+\nu) \div \bv 
+\overline{P(\rho,\theta)}.
\end{equation}
Note that we are able to control  integrals $\int_\Omega G_\epsilon dx = \int_\Omega P(\rho_\epsilon,\theta_\epsilon) dx$ and 
$\int_\Omega G dx= \int_\Omega  \overline{P(\rho,\theta)} dx$, where $\overline{P(\rho,\theta)} = \overline{P_b(\rho)} + \theta \big(\overline{\int_0^\rho
K(t)dt}\big)$.

The result of the lemma below gives the most important properties of the effective viscous flux, guaranteeing the compactness of 
$\{G_\epsilon\}$ as well as the pointwise bound
of the limit in term of the parameter $k$ from definition (\ref{def-K}).
\smallskip

\begin{lemma} \label{l 4.4} 
We have, up to a subsequence $\epsilon \to 0^{+}$:
\begin{equation}\label{e26}
G_\epsilon \to G \mbox{ strongly in } L_2(\Omega) 
\end{equation}
and
\begin{equation} \label{e26a}
||G||_{L_\infty} \leq C(\eta)(1+k^{1+\frac 23 \gamma +\eta}) \mbox{ \ \ \ for any $\eta>0$.}
\end{equation}
\end{lemma}

\begin{proof} 
 The function $G_\epsilon$ can be naturally decomposed as
\begin{equation*}\label{e27}
G_\epsilon=G^1_\epsilon +G^2_\epsilon,
\end{equation*}
where $\int_\Omega G^2_\epsilon dx=0$ and
$\nabla G_\epsilon^2=-\frac 12 \epsilon \Delta \rho_\epsilon \bv_\epsilon- \mu \rot \,\bomega_\epsilon^2$.
Thus
\begin{equation*}\label{e28}
||G_\epsilon^2||_{L_q(\Omega)}\leq C(\epsilon ||\Delta \rho_\epsilon \bv_\epsilon||_{W^{-1}_q(\Omega)}+ 
\mu \|\rot\, \bomega_\epsilon^2\|_{W^{-1}_q(\Omega)}).
\end{equation*}
Using Lemma \ref{l 4.3} we see that
$$
\| G_\epsilon^2\|_{L_q(\Omega)} \leq C(k) \epsilon^{\frac 12}, \qquad 1\leq q\leq 2.
$$
Next, using again Lemma \ref{l 4.3} and calculations in its proof, we immediately see that (recall that $|\int_\Omega G_\epsilon dx|\leq
C$)
\begin{equation}\label{e30}
||G^1_\epsilon||_{W^1_q(\Omega)} \leq C(1+k^{1+\gamma (\frac 43 -\frac 2q)})
\qquad \mbox{ for } 2\leq q \leq 3m.
\end{equation}
Thus we have, at least for a subsequence
$$
G_\epsilon^1 \to G^1 \qquad \mbox{ in } L^\infty(\Omega) \mbox{ \ \ and \ \ }
G_\epsilon^2 \to 0 \mbox{ \ in \ } L_2(\Omega).
$$
Therefore
$$
G_\epsilon =G_\epsilon^1 + G_\epsilon^2 \to G^1 \qquad \mbox { in } L^q(\Omega), \qquad 1\leq q \leq 2
$$
and due to the definition, $G^1=G$. Finally, choosing $q=3+\tilde{\eta}$ in (\ref{e30})
$$
\|G\|_{L_\infty(\Omega)} \leq C(q)\|G\|_{W^1_q(\Omega)} \leq C(q) \sup_{\epsilon>0}\|G^1_\epsilon\|_{W^1_q(\Omega)} \leq C(\eta)(1+ k^{1+ \frac 23 \gamma + \eta})
$$
with $\eta>0$, arbitrarily small if $\tilde \eta$ is so. This finishes the proof of Lemma \ref{l 4.4}.
\end{proof}

\section{Limit passage}

In this section we apply the properties of the effective viscous flux shown in the previous part. First we prove a result
characterizing the sequence of approximative densities.

\begin{theorem}\label{t3}
{\it  There exits a sufficiently large number $k_0>0$ such that for $k> k_0$
\begin{equation}\label{s9}
\frac{k-3}{k}(k-3)^\gamma -||G||_{L_\infty(\Omega)} \geq 1
\end{equation}
and for a subsequence $\epsilon \to 0^{+}$ it holds
\begin{equation}\label{s2}
\lim_{\epsilon \to 0^+} |\{ x \in \Omega: \rho_\epsilon(x) > k-3 \} | =0.
\end{equation}
In particular it follows: $\overline{K(\rho)\rho}=\rho$ a.e. in $\Omega$.}
\end{theorem}
\smallskip

\begin{proof}  We define a smooth function
$M:\R^{+}_0 \to [0,1]$ such that
\begin{equation*}\label{s3}
M(t)=\left\{
\begin{array}{lcr}
1 & \mbox{ for } & t \leq k-3 \\
\in [0,1] & \mbox{ for } & k-3 < t < k-2 \\
0 & \mbox{ for } & k-2 \leq t
\end{array}
\right.
\end{equation*}
and  $M'(t)<0$ for $t \in (k-3,k-2)$.

We follow the method  introduced in \cite{MuPo}. First we multiply the approximative continuity equation (\ref{ap-ns})$_1$
by $M^l(\rho_\epsilon)$ for $l \in \mathbb N$ getting
\begin{equation*}\label{s4}
\int_\Omega \left( \int_0^{\rho_\epsilon(x)}\,t\,l\,M^{l-1}(t)M'(t)dt\right)\div \bv_\epsilon \geq R_\epsilon
\end{equation*}
with $R_\epsilon \to 0$ as $\epsilon \to 0$, as
\begin{equation*}\label{s5}
\epsilon \int_\Omega M^l(\rho_\epsilon) \Delta \rho_\epsilon dx = -\epsilon l \int_\Omega
M^{l-1}(\rho_\epsilon) M'(\rho_\epsilon) |\nabla \rho_\epsilon|^2 dx \geq 0.
\end{equation*}

Next, recalling  definitions of $G_\epsilon$ and $M$, we obtain
\begin{equation*}\label{s6}
\begin{array}{c}
\displaystyle
-(k-3) \int_\Omega\Big(\int_0^{\rho_\epsilon(x)} l M^{l-1}(t)M'(t)dt\Big) P(\rho_\epsilon,\theta_\epsilon) dx
\\[8pt]
\displaystyle
\leq k \Big|\int_\Omega \Big(\int_0^{\rho_\epsilon (x)} -l M^{l-1}(t)M'(t) dt\Big) G_\epsilon dx \Big| +R_\epsilon.
\end{array}
\end{equation*}

Thus the properties of  $M$ lead us to the following inequality
\begin{equation*}\label{s7}
\begin{array}{c}
\displaystyle\frac{k-3}{k} \int_{\{\rho_\epsilon > k-3\}}(1-M^l(\rho_\epsilon))P(\rho_\epsilon,\theta_\epsilon)dx
\leq \int_{\{\rho_\epsilon > k-3\}}(1-M^l(\rho_\epsilon))|G_\epsilon|dx + |R_\epsilon|.
\end{array}
\end{equation*}

From the explicit form of the pressure function  (\ref{ap-P}) we find
\begin{equation*}\label{s8}
\begin{array}{c}
\displaystyle \frac{k-3}{k}(k-3)^\gamma |\{\rho_\epsilon > k-3\}|-\frac{k-3}{k}||P(\rho_\epsilon,\theta_\epsilon)||_{
L_2(\Omega)}||M^l(\rho_\epsilon)||_{L_2(\Omega)}
\\[12pt]
\displaystyle \leq ||G||_{L_\infty(\Omega)}|\{\rho_\epsilon > k-3\}|+
||G-G_\epsilon||_{L_1(\Omega)}+|R_\epsilon|.
\end{array}
\end{equation*}

But by Lemma \ref{l 4.4} -- the inequality $(\ref{e26a})$ -- we are able to choose  $k_0$ so large that
for all $k>k_0$ we have (\ref{s9}),
since $\gamma > 3$ and $||G||_{L_\infty(\Omega)}\leq C_\eta(1+k^{1+\frac{2}{3} \gamma +\eta})$
with $0<\eta\leq \frac{\gamma-3}{6}$.

Hence we get
\begin{equation}\label{s10}
|\{x \in \Omega: \rho_\epsilon(x) > k-3\}|\leq
C\left(||M^l(\rho_\epsilon)||_{L_2(\{\rho_\epsilon>k-3\})}+||G-G_\epsilon||_{L_1(\Omega)}+|R_\epsilon|\right).
\end{equation}

Now, let us fix $\delta>0$. Then there exists $\epsilon_0>0$ such that for $\epsilon < \epsilon_0$ 
\begin{equation}\label{s11}
C(||G-G_\epsilon||_{L_1(\Omega)}+|R_\epsilon|)\leq \delta/2.
\end{equation}
Having  $\epsilon$ fixed, we consider the sequence 
$\{M^l(\rho_\epsilon)I_{\{\rho_\epsilon>k-3\}}\}_{l\in \mathbb N}$,
where $I_A$ is the characteristic function of a set $A$.
We see that
it monotonely pointwise converges to zero. Thus by the Lebesgue theorem we are able to 
find $l=l(\epsilon,\delta)$ such that 
\begin{equation}\label{s12}
C||M^l(\rho_\epsilon)||_{L_2(\{\rho_\epsilon>k-3\})}\leq \delta/2.
\end{equation}
From (\ref{s10}), (\ref{s11}) and (\ref{s12})  we obtain
\begin{equation}\label{s13}
\lim_{\epsilon \to 0} |\{x\in \Omega; \rho_\epsilon(x)>k-3\}|\leq \delta.
\end{equation}
As $\delta>0$ can be chosen arbitrarily small, Theorem \ref{t3} is proved.
\end{proof}

Thanks to Theorem \ref{t3} we are prepared to present the main part of the proof, i.e. the 
pointwise convergence of the density.

\begin{lemma}
 We have
\begin{equation}\label{s14}
\int_\Omega \overline{P(\rho,\theta)\rho}dx \leq \int_\Omega G \rho dx 
\mbox{ \ \ and \ \ } \int_\Omega \overline{P(\rho,\theta)}\rho dx=\int_\Omega G \rho dx
\end{equation}
consequently, $\overline{P(\rho,\theta)\rho}=\overline{P(\rho,\theta)}\rho$ and
up to a subsequence $\epsilon \to 0^+$
\begin{equation}\label{s15}
 \rho_\epsilon \to \rho \mbox{ \ \ strongly in } L_q(\Omega) 
\mbox{ \ \ for any \ \ } q<\infty.
\end{equation}

\end{lemma}

\begin{proof}
Due to Theorem  \ref{t3} we are able to omit $K(\rho)$ in the limit equation. For details we refer to \cite{MuPo} -- section 4,
consideration for (4.16). 

Examine the approximative continuity equation (\ref{ap-ns})$_1$. We use as test function $\ln (\rho_\epsilon+\delta)$ and passing with $\delta\to 0^+$
we obtain
\begin{equation}\label{s16}
\int_\Omega K(\rho_\epsilon) \bv_\epsilon \cdot \nabla \rho_\epsilon dx \geq \epsilon C(k),
\end{equation}

thus Theorem \ref{t3} implies
\begin{equation}\label{s17}
-\int_\Omega \rho_\epsilon \div \bv_\epsilon dx \geq R_\epsilon.
\end{equation}
Applying (\ref{evf_eps}) to (\ref{s17}), passing with $\epsilon \to 0$, then  by the strong convergence of 
$G_\epsilon$ -- see $(\ref{e26})$ -- we conclude that
$\overline{G \rho}=G \rho$, so the first relation in (\ref{s14}) is proved.

Next we consider the limit to the continuity equation, i.e. $\div (\rho \bv)=0$. Testing it by $\ln \rho$
with an application of Friedrich's lemma to have possibility to use test functions with lower regularity we obtain (for details see
\cite{MuPo})
\begin{equation*}\label{s18}
\int_\Omega \rho \div\bv  dx =0.
\end{equation*}
The definition of $G$ -- (\ref{evf}) -- shows the second part of (\ref{s14}).

Due to elementary properties of weak limits we get
$\rho \overline{P(\rho,\theta)} \leq \overline{P(\rho,\theta)\rho}$ a.e. in $\Omega$,
but (\ref{s14}) implies $\int_\Omega (\overline{P(\rho,\theta)\rho}- \overline{P(\rho,\theta)} \, \rho )dx \leq 0$,
hence
\begin{equation*}\label{s19}
\rho \overline{P(\rho,\theta)} =\overline{P(\rho,\theta)\rho}
\mbox{ \ \ a.e.,\ \ \ \ \ \ \ i.e. \ \ }
\overline{\rho^{\gamma+1}}+\overline{\rho^2}\theta=\overline{\rho^\gamma}\rho+\rho^2\theta \mbox{ \ \ a.e.}
\end{equation*}
However,  $\overline{\rho^{\gamma+1}}\geq \overline{\rho^\gamma}\rho$ and $\overline{\rho^2}\theta
\geq \rho^2\theta$, so
\begin{equation*}\label{s20}
\overline{\rho^{\gamma+1}}= \overline{\rho^\gamma}\rho \mbox{ a.e. \ \ \ \ \ and 
\ \ \ \ \ } \overline{\rho^2}\theta
= \rho^2\theta \mbox{ \ \ a.e.}
\end{equation*}
By Lemma \ref{l 4.2}   the temperature $\theta>0$ a.e., we conclude $\overline{\rho^2}=\rho^2$ and for a suitably taken subsequence 
\begin{equation}\label{s21}
\lim_{\epsilon \to 0}||\rho_\epsilon -\rho||^2_{L_2}=\overline{\rho^2}-\rho^2=0.
\end{equation}
Thus the limit (\ref{s21}) implies $\rho_\epsilon \to \rho$ strongly in $L_2(\Omega)$ 
 and by the pointwise boundedness of $\rho_\epsilon$ and  $\rho$
we conclude (\ref{s15}).
\end{proof}

Next, we would like to study the limit of the energy equation. 
The first observation concerns the velocity,
we obtain  
the strong convergence of its gradient.

Recall that from Theorem \ref{t3} and due to the strong convergence of the temperature it follows
\begin{equation*}\label{h1}
P(\rho_\epsilon, \theta_\epsilon) \to p(\rho,\theta)\mbox{ \ \  strongly in } L_2(\Omega),
\end{equation*}
hence (\ref{e26}) implies
\begin{equation}\label{h2}
\div \bv_\epsilon \to \div \bv \mbox{  \ \ \ \ strongly in } L_2(\Omega).
\end{equation}
Additionally we already proved that
\begin{equation}\label{h3}
\rot \bv_\epsilon \to \rot \bv \mbox{ \ \ \ \   strongly in } L_2(\Omega),
\end{equation}
since we observed that the vorticity can be written as sum of two parts, one bounded in $W^1_q(\Omega)$ and the other one going 
strongly to zero
in $L_2(\Omega)$. 

The regularity of systems (\ref{e5}) and (\ref{e6}) and convergences (\ref{h2}) and (\ref{h3})
 imply immediately that
\begin{equation*}\label{h4}
\bv_\epsilon \to \bv \mbox{ \ \ \ \ strongly in }H^1(\Omega).
\end{equation*}

In particular, we get
\begin{equation}\label{h5}
S(\bv_\epsilon):\nabla \bv_\epsilon \to S(\bv) :\nabla \bv \mbox{ \ \ \ strongly  at least in } L_1(\Omega).
\end{equation}

This fact will be crucial in considerations for the limit of the energy equation.
Recall that 
\begin{equation}\label{h6}
\begin{array}{c}
\rho_\epsilon \to \rho \mbox{ in } L_q(\Omega) \mbox{ \ \ for } q <\infty,
\\
\bv_\epsilon \to \bv \mbox{ in } W^1_q(\Omega) \mbox{ \ \ for \ } q < 3m,
\\
\theta_\epsilon \to \theta \mbox{ in } L_{q}(\Omega) \mbox{ \ for \ } q<3m,
\\
\theta_\epsilon \rightharpoonup \theta \mbox{ in } W^1_{\min\{2,\frac{3m}{m+1}\}}(\Omega).
\end{array}
\end{equation}
Consider the weak form of $(\ref{ap-ns})_3$. For a smooth function $\phi$ we have
\begin{equation}\label{h7}
\begin{array}{c}
\displaystyle\int_\Omega (1+\theta_\epsilon^m)\frac{\epsilon +\theta_\epsilon}{\theta_\epsilon} \nabla \theta_\epsilon \cdot
\nabla \phi dx
+\int_{\partial \Omega} L(\theta_\epsilon)(\theta_\epsilon-\theta_0)\phi d \sigma
\\[12pt]
\displaystyle 
-\int_\Omega \left[\left(\int_0^{\rho_\epsilon(x)} K(t)dt\right) \bv_\epsilon \cdot\nabla (\theta_\epsilon \phi)  +
K(\rho_\epsilon)\rho_\epsilon \bv_\epsilon \cdot\nabla (\theta_\epsilon \phi)\right]dx
\\[12pt]
\displaystyle +
\int_\Omega \left[ K(\rho_\epsilon)\rho_\epsilon \bv_\epsilon\cdot \nabla \theta_\epsilon \phi +
\div (\theta_\epsilon \bv_\epsilon \phi) \int_0^{\rho_\epsilon(x)} K(t)dt\right]dx=
\int_\Omega \bS(\bv_\epsilon):\nabla \bv_\epsilon \phi dx.
\end{array}
\end{equation}
Thanks to (\ref{h6}),
\begin{equation*}\label{h11}
(1+\theta_\epsilon^m)\frac{\epsilon + \theta_\epsilon}{\epsilon}\nabla \theta_\epsilon
\rightharpoonup 
(1+\theta^m)\nabla \theta \mbox{  \ \ \ in } L_1(\Omega).
\end{equation*}
Passing to the limit with the last four terms of the l.h.s. of (\ref{h7}) we get
\begin{equation}\label{h12}
\begin{array}{c}
\displaystyle
\int_\Omega \left[- \rho \bv \nabla (\theta \phi) - \rho \bv \nabla (\theta \phi)
+\rho \phi \bv \nabla \theta + \div(\theta \phi \bv) \rho\right]dx
\\[8pt]
\displaystyle
=\int_\Omega \left[- \rho \theta \bv \cdot\nabla \phi + \rho \theta \div \bv \phi\right] dx.
\end{array}
\end{equation}
In (\ref{h12}) we essentially used  the strong convergence of the density.

To control the behavior of the boundary term we note that due to $(\ref{h6})_4$ we see that
$\theta_\epsilon|_{\partial \Omega} \to \theta|_{\partial \Omega}$ strongly in $L_{l+1}(\partial\Omega)$.
Thus recalling (\ref{h5})
 we get  at the limit

\begin{equation}\label{h13}
\begin{array}{c}
\displaystyle
\int_\Omega (1+\theta^m)\nabla \theta \cdot \nabla \phi dx + \int_{\partial \Omega}
L(\theta)(\theta-\theta_0)d\sigma-
\int_\Omega \rho \theta \bv \cdot \nabla \phi dx
\\[8pt]
\displaystyle =
\int_\Omega \bS(\bv):\nabla \bv \phi dx - \int_\Omega \rho \theta \div \bv \phi dx.
\end{array}
\end{equation}

To conclude, note that we may show that the limit functions $\theta$ and $\bv$ belong to $W^1_p(\Omega)$ for any $p<\infty$. To see this,
we introduce the function $\Phi(\theta) = \int_0^\theta (1+ t^m) dt$, similarly as in Section 3, formula (\ref{D1}). Thus from (\ref{h13}) we immediately see
that $\theta \in L_\infty(\Omega)$ and $\bv \in W^1_p(\Omega)$ for any $p<\infty$. Using this fact once more in the energy equation,
we observe that $\theta \in W^1_p(\Omega)$, $p<\infty$. Theorem \ref{t1} is proved.
 
 \bigskip
 
{\footnotesize {\bf Acknowledgement.} The work has been granted by the working program
between Charles and Warsaw Universities. The first author has been partly supported by  the Polish KBN grant
No. 1 P03A 021 30.
 The work of the second author is a part of
the research project MSM 0021620839 financed by MSMT and partly supported by the grant of
the Czech Science Foundation No. 201/05/0164 and by the project LC06052 (Jind\v{r}ich Ne\v{c}as Center for Mathematical Modeling).

\end{document}